\documentclass[11pt,a4paper]{article}
\usepackage{amsmath}
\usepackage{amsthm}
\usepackage{amssymb}
\usepackage{amscd}
\usepackage{graphicx}
\usepackage{epsfig}
\usepackage[matrix,arrow,curve]{xy}
\usepackage{color}
\usepackage{mathrsfs}

\usepackage{bbm}
\usepackage{stmaryrd}
\usepackage{hyperref}
\usepackage{multirow}

\usepackage{latexsym}
\usepackage{amsfonts}

\input xy
\mathchardef\za="710B  
\mathchardef\zb="710C  
\mathchardef\zg="710D  
\mathchardef\zd="710E  
\mathchardef\zve="710F 
\mathchardef\zz="7110  
\mathchardef\zh="7111  
\mathchardef\zvy="7112 
\mathchardef\zi="7113  
\mathchardef\zk="7114  
\mathchardef\zl="7115  
\mathchardef\zm="7116  
\mathchardef\zn="7117  
\mathchardef\zx="7118  
\mathchardef\zp="7119  
\mathchardef\zr="711A  
\mathchardef\zs="711B  
\mathchardef\zt="711C  
\mathchardef\zu="711D  
\mathchardef\zvf="711E 
\mathchardef\zq="711F  
\mathchardef\zc="7120  
\mathchardef\zw="7121  
\mathchardef\ze="7122  
\mathchardef\zy="7123  
\mathchardef\zf="7124  
\mathchardef\zvr="7125 
\mathchardef\zvs="7126 
\mathchardef\zf="7127  
\mathchardef\zG="7000  
\mathchardef\zD="7001  
\mathchardef\zY="7002  
\mathchardef\zL="7003  
\mathchardef\zX="7004  
\mathchardef\zP="7005  
\mathchardef\zS="7006  
\mathchardef\zU="7007  
\mathchardef\zF="7008  
\mathchardef\zW="700A  
\newcommand{\be}{\begin{equation}}
\newcommand{\ee}{\end{equation}}
\def\bel#1\eel{\begin{align}#1\end{align}}
\def\bes#1\ees{\begin{align*}#1\end{align*}}

\newcommand{\raa}{\rightarrow}

\newcommand{\bea}{\begin{eqnarray}}
\newcommand{\eea}{\end{eqnarray}}
\newcommand{\beas}{\begin{eqnarray*}}
\newcommand{\eeas}{\end{eqnarray*}}

\newcommand{\Z}{\mathbb{Z}}
\newcommand{\R}{\mathbb{R}}

\newcommand{\C}{\mathbb{C}}

\newcommand{\Pe}{\mathbb{P}}
\newcommand{\pa}{\partial}
\newcommand{\ti}{\times}

\def\cN{{\mathcal N}}

\def\cH{{\mathcal H}}
\def\cO{{\mathcal O}}

\def\cM{{\mathcal M}}
\def\cN{{\mathcal N}}

\def\cF{{\mathcal F}}

\def\cU{{\mathcal U}}

\def\sD{{\mathsf D}}

\def\sJ{{\mathsf J}}

\def\sT{{\mathsf T}}
\def\sV{{\mathsf V}}

\def\sS{{\mathsf S}}

\def\sj{{\mathsf j}}

\def\s{\mathsf{s}}

\def\br{{\mathbf r}}

\newcommand{\bk}[2]{\ensuremath{\langle #1 , #2\rangle}}

\newcommand{\la}{\langle}
\newcommand{\ran}{\rangle}

\def\Hom{\mathsf{Hom}}

\newcommand{\mn}{{\medskip\noindent}}

\newcommand{\no}{{\noindent}}
\newcommand{\xd}{\textnormal{d}}
\newcommand{\half}{{\frac{1}{2}}}

\newcommand{\nn}{\nonumber}
\newcommand{\ot}{\otimes}
\newcommand{\op}{\oplus}
\newcommand{\we}{\wedge}
\def\wt{\widetilde}
\def\ol{\overline}

\def\Sec{\operatorname{Sec}}
\def\Rt{{\R^\ti}}
\def\Lt{{L^\ti}}

\def\dts{\xd_{\sT^*}}

\def\on{\operatorname}

\def\p{\mathbf{p}}
\def\wh{\widehat}

\def\s{\mathfrak{s}}

\newcommand{\Ll}{{\pounds}}

\def\n{\nabla}
\def\sgn{\on{sgn}}
\def\tO{\widetilde{\cO}}
\def\tU{\widetilde{\cU}}
\def\tM{{\wt M}}
\def\tzm{{\wt\zm}}
\def\tP{{\wt P}}
\def\tC{{\wt C}}

\def\Bt{{B^\times}}
\newcommand{\id}{\mathrm{id}}
\newcommand{\nm}[1]{\ensuremath{\Vert #1 \Vert}}

\def\cm{\color{magenta}}

\def\tzh{{\wt\zh}}
\def\tA{{\wt A}}

\def\tzx{{\wt\zx}}

\def\pc{{\phi_C}}

\def\bp{{\bar\phi_C}}

\def\tbp{{\bar\phi_{\tC}}}

\def\GL{\operatorname{GL}}
\def\tip{{\ti^!}}
\def\hX{{X^\s}}
\def\hm{\wh{\zm}}
\def\hg{{\wh{g}_M}}
\def\bl{\big( }
\def\br{\big) }
\def\Bl{\Big( }
\def\Br{\Big) }
\def\gc{{{\bar g}_C}}
\def\fn{{\mathfrak N}}

\newtheorem{theorem}{Theorem}[section]
\newtheorem{proposition}[theorem]{Proposition}

\newtheorem{corollary}[theorem]{Corollary}

\theoremstyle{definition}
\newtheorem{example}[theorem]{Example}
\newtheorem{remark}[theorem]{Remark}
\newtheorem{definition}[theorem]{Definition}

\begin{document}
\title{Sasaki structures\\ on general contact manifolds\footnote{Research funded by the National Science Centre (Poland) within the project WEAVE-UNISONO, No. 2023/05/Y/ST1/00043.}}
\author{
Katarzyna Grabowska\footnote{email:konieczn@fuw.edu.pl }\\
\textit{Faculty of Physics,
                University of Warsaw}
\\ \\
Janusz Grabowski\footnote{email: jagrab@impan.pl} \\
\textit{Institute of Mathematics, Polish Academy of Sciences}
\\ \\
Rouzbeh Mohseni\footnote{email: rouzbeh.iii@gmail.com}\\
\textit{Faculty of Mathematics and Computer Science, University of Łódź}}

\date{}
\maketitle
\begin{abstract}
We extend the notion of a Sasakian structure from the classical setting of a cooriented contact manifold, where it is given by a compatibility between a contact form $\zh$ and a Riemannian metric $g_M$ on $M$, to the case of an arbitrary contact structure understood as a contact distribution. In the cooriented case, this compatibility can be equivalently expressed by the fact that the symplectic form {$\omega=\xd(s^2\zh)/2$} and the cone metric $g(x,s)=\xd s\ot\xd s+s^2g_M(x)$ define a K\"ahler structure on the cone $\cM=M\times\R_+$.

Since general contact structures admit canonical realizations as homogeneous symplectic structures $\zw$ on principal $\Rt$-bundles $P\to M$, it is natural to interpret Sasakian geometry in full generality in terms of suitable homogeneous K\"ahler structures on $P$. We characterize homogeneous K\"ahler structures on symplectizations $(P,\zw)$ associated with arbitrary contact structures on $M$, and show that they canonically determine a two-sheeted covering $\tM$ of $M$ equipped with a contact form. This reduces the problem to the cooriented case and leads to a notion of a \emph{generalized Sasakian structure} on $M$ associated with homogeneous K\"ahler structures on $(P,\zw)$.

Moreover, since products of K\"ahler manifolds are again K\"ahler, our framework naturally yields a concept of a product of Sasakian manifolds. {All} constructions are intrinsic and conceptual, avoiding any \emph{ad hoc} choices.

\par


\end{abstract}
\section{Introduction}
The origins of contact geometry can be traced back to the works of Huygens, Barrow, and Newton. Although the theory of contact transformations had already been developed by Sophus Lie \cite{Lie:1890}, the modern theory of contact structures began with the influential paper of Boothby and Wang \cite{Boothby:1958} in 1958. One year later, Gray \cite{Gray:1959} introduced the notion of an \emph{almost contact structure}. Today, a contact structure is understood as a hyperplane distribution $C\subset \sT M$ on an odd-dimensional manifold $M$ of dimension $2n+1$, locally given as the kernel of a contact form $\zh$, i.e., a $1$-form satisfying $\zh\wedge (\xd\zh)^n\neq 0$.

If such a form exists globally, the contact structure is called \emph{coorientable} or \emph{trivializable}. For general background on contact geometry, we refer to the monograph by Geiges \cite{Geiges:2008}. Allowing nontrivializable contact structures is essential, since many natural and canonical examples, such as contact structures on first jet bundles of line bundles, are not globally coorientable. {Nevertheless, every contact manifold admits { a canonical coorientable double covering, carrying a (non-canonical) global contact for; see Corollary~\ref{cr1}.}

\mn In \cite{Grabowski:2013}, the second author introduced the notion of a \emph{symplectic $\R^\times$-bundle} $(P,\zw)$ over a manifold $M$, where $\Rt$ denotes the multiplicative group of nonzero real numbers, and proved that symplectic $\R^\times$-bundles are equivalent to contact manifolds (see also \cite{Arnold:1989}). The symplectic $\R^\times$-bundles associated with a contact manifold $(M,C)$ are all mutually isomorphic and are called the \emph{symplectic covers} of $(M,C)$. We briefly recall these constructions in Section~\ref{Sec3}. Later, in a series of papers \cite{Grabowska:2022,Grabowska:2023,Grabowska:2023a}, this viewpoint was applied to geometric mechanics, based on the principle that contact geometry should be understood primarily in terms of the homogeneous symplectic form $\zw$, rather than as an odd-dimensional analogue of symplectic geometry. Here, homogeneity means that $h_s^*(\zw)=s\zw$, where $s\mapsto h_s$ denotes the principal $\Rt$-action on $P$. This perspective considerably clarifies and unifies the relation between contact and symplectic geometry.

\mn Symplectic and complex structures can both be described as integrable $G$-structures. In contrast, the situation for contact structures is more subtle, and the notion of integrability must be modified accordingly; see, e.g., \cite{Tortorella:2020}. In 1961, Sasaki {and Hatakeyama} \cite{Sasaki:1961} observed that an almost contact metric structure on a manifold $M$ naturally induces an almost complex structure on the product manifold $M\times\R$, equivalently on the cone $\cM=M\times\R_+$. This made it possible to study contact geometry through complex differential geometry on the cone and led to the notion now known as a \emph{Sasakian structure}. A particularly transparent description starts from a metric contact structure $(M,\zh,g_M)$ and considers the cone $\cM=M\times\R_+$ endowed with the metric $g=\xd s\ot \xd s+s^2g_M$
and the symplectic form $\zw=\xd(s^2\zh)$. Here, $s$ denotes the standard coordinate on the $\R_+$ factor.
Then $(M,\zh,g_M)$ is Sasakian if and only if $(\cM,\zw,g)$ is a K\"ahler manifold.

\mn The above construction uses a symplectic form $\zw=\xd(s^2\zh)$ which is $2$-homogeneous, $h_s^*(\zw)=s^2\zw$, together with a compatible metric $g$ which is likewise $2$-homogeneous with respect to the principal $\R_+$-bundle structure on $\cM$. However, from the viewpoint of symplectic covers, the natural homogeneous object is rather the $1$-homogeneous symplectic form $\zw'=\xd(s\zh)$. Correspondingly, one is led to consider the $1$-homogeneous tensor
$g'=(\xd s\ot\xd s)/s+sg_M$.
This shift is essential because the canonical symplectic form on the symplectic cover $\zt:P\to M$ of a contact manifold is always $1$-homogeneous. On the other hand, genuinely $1$-homogeneous symmetric tensors on $\Rt$-bundles cannot be Riemannian, since multiplication by $-1\in\Rt$ changes their sign. To overcome this difficulty, we introduce the notion of {an} \emph{(almost) K\"ahlerian $\Rt$-bundle}, where the symplectic form is homogeneous while the metric is only positively homogeneous.

\mn We show that such homogeneous K\"ahlerian structures are naturally associated with a Riemannian metric $g_M$ on the contact manifold together with a principal connection on the bundle $P$. Moreover, the metric $g_M$ canonically determines a \emph{calibration} on $P$, i.e., a positive homogeneous function. This calibration provides an ``almost trivialization'' of the contact structure, since it locally determines a contact form $\pm\zh$. The K\"ahler structure on $P$ then induces a compatibility condition between $\pm\zh$ and $g_M$, encoded by a CR structure on $(M,\pm\zh)$, which allows one to express $g_M$ as a generalized Levi metric. In this way, we obtain a natural extension of Sasakian geometry to arbitrary contact manifolds, called \emph{generalized Sasaki manifolds}. {As far} as the terminology is concerned, let us note that the name `generalized Sasakian structures' has different meaning in a paper by Ken’ich{ i} Sekiya \cite{Sekiya:2015}, which is related to the concept of generalized complex geometry by Hitchin, and that we use the term `homogeneous K\"ahler manifold' in relation to homogeneity structures, not to K\"ahler structures on homogeneous spaces--quotients of Lie groups.

\mn Summarizing, in this paper we introduce the notion of a \emph{generalized Sasakian manifold}, derived from the principle of homogeneity of a K\"ahler structure on the symplectic cover of a contact manifold $(M,C)$. This fits into the broader philosophy of describing geometric structures via homogeneity, in the same spirit as viewing contact structures as homogeneous symplectic structures and Jacobi structures as homogeneous Poisson structures. The resulting category of generalized Sasakian manifolds has good functorial and product properties, leading in particular to a natural notion of a \emph{Sasakian product}. We illustrate the theory with examples ranging from classical Sasakian manifolds to a genuinely nontrivial example related to the M\"obius band.

\section{Line and $\R^\ti$-principal bundles}\label{Sec2}
Vector bundles with one-dimensional fibers will be called \emph{line bundles}. If $\zt: L\rightarrow M$ is a line bundle over a manifold $M$, then the submanifold $P=L^\times\subset L$ of nonzero vectors, where $L^\times=L\setminus 0_M$, is canonically a principal bundle over $M$ with the structure group $(\R^\times, \cdot)=\on{GL}(1;\R)$, i.e., the group of invertible reals with multiplication. The $\R^\ti$-action on $L^\times$ comes from the multiplication by reals in $L$. If $(x^i)$ are local coordinates in $U\subset M$ and $(x^i,t)$ are affine coordinates in $\zt^{-1}(U)\subset L$, associated with a local trivialization $\zt^{-1}(U)\simeq U\ti\R$,
then to distinguish $L$ from $\Lt$ we will use local coordinates $(x^i,s)$ in $L^\ti$, where $s$ is the restriction of the function $t$ to {$L^\ti$}. The $\Rt$-action reads $h_\zn(x^i,s)=(x^i,\zn\cdot s)$. Actually, any principal $\Rt$-bundle $P$ is of the form $L_P^\ti$, where $L_P$ is the canonical line bundle associated with $P$.
\begin{remark}\label{prodlb}
{With the described equivalence of line and $\Rt$-bundles, we can consider the \emph{category of line bundles}, i.e., the non-full subcategory of the category of vector bundles where objects are line bundles, and morphisms are isomorphisms on fibers.}
\end{remark}
\begin{example}\label{MB}
The trivial bundles are $L=M\ti\R$ and $P=\Lt=M\ti\Rt$. Probably, the simplest example of a line bundle that is not trivializable is that of the M\"obius band.
The M\"obius band, as a line bundle $B\to S^1$, can be described by two charts. We take $$\mathcal{O}=\{(x,t)\in\R^2: x\in]0,1[\}$$
and
$$\mathcal{U}=\{(x,t)\in\R^2: x\in]1/2,3/2[\}.$$
Our M\"obius band is the topological space $B$ obtained by gluing these two strips by a local homeomorphism
$$\zF:\cO\supset \{(x,t)\in\cO: x\ne 1/2\}\to\{(x,t)\in\cU: x\ne 1\}\subset\cU,$$
which reads
\be\label{tmm}\zF(x,t)=
\begin{cases}
(x,t)\quad\text{if}\quad x\in]1/2,1[\\
(x+1,-t)\quad\text{if}\quad x\in]0,1/2[\,.
\end{cases}
\ee
Hence, we can view $\cO$ and $\cU$ as coordinate charts in $B$, and $\zF$ as the corresponding transition map which, clearly, turns $B$ into a smooth manifold. It is easy to see that $B$ is a line bundle $B\to S^1=\R/\Z$, with the projection induced by $(x,t)\to x$ in the charts $\cO$ and $\cU$. These charts give us local trivializations over $S^1$ without a point.

\mn {This line bundle is trivializable if and only if it admits a global nonvanishing section $\zs:S^1\to B$}; suppose it exists. Then, in the chart $\cO$, the section $\zs$ is represented by a function $F_\cO:]0,1[\to\R$ which is positive or negative. Suppose the positivity. In $\cU$, the section $\zs$ is represented by a nonvanishing function $F_\cU:]1/2,3/2[\to\R$. But due to the form of the transition map $\zF$, the function $F_\cU$ is $F_\cO$ on $]1/2,1[$ and $F_\cU(x)=-F_\cO(x-1)$ on $]1,3/2[$, so it vanishes at some point. Of course, we can use the same charts for $\Bt$, with the only difference that $t\ne 0$.
\end{example}
\mn For $\Rt$-bundles (or $\R_+$-bundles) $P$ we have a natural concept of homogeneity.
\begin{definition}
Let $\zt:P\to M$ be a principal $G$-bundle, where $G=\Rt$ or $G=\R_+$, with the $G$-action $s\mapsto h_s$.
A tensor field $K$ on $P$ we call \emph{homogeneous} of degree $k\in\Z$ if
\[h_s^*(K)=s^k\cdot K \quad\text{for all}\quad s\in G.
\]
Here, $h_s^*(K)$ is the pullback of the tensor field $K$ associated with the diffeomorphism $h_s$. Covariant tensors that are 1-homogeneous we will call just \emph{homogeneous}.
\end{definition}
\no{This definition can be generalized for any $G$ by means of a group homomorphism $f:G\to\Rt$ and \emph{$f$-homogeneity}. We will use $f:\Rt\to\Rt$ of the form $|s|$ and $\sgn(s)$ later on.}

\begin{example}
On the trivial $\Rt$-bundle $P=M\ti\Rt$, homogeneous functions are of the form $A(x)s$ and homogeneous 1-forms read $A(x)\xd s+s\za(x)$, where $s$ is the fiber coordinate, $A$ is a function on $M$ and $\za$ is a 1-form on $M$.
\end{example}
\no The $\Rt$-action $h:\Rt\ti P\to P$ on an $\Rt$-principal bundle $\zt:P\to M$ can be lifted to a principal action on the cotangent bundle $\sT^\ast P$, which we call the \emph{phase lift} and denote $\dts h$. Note, however, that this lift is not the standard cotangent lift of a group action. It is defined by the formula {(see \cite[Section 5]{Grabowski:2009}).}
$$({\dts}h)_\zn=\zn\cdot(\sT h_{\zn^{-1}})^\ast.$$

\mn Like for homogeneity, we can generalize it to any $G$-bundle and the lift $h_\zn^f$, multiplying not just by $\zn$ but by $f(\zn)$, for a group homomorphism $f:G\to\Rt$. However, the above phase lift is rather particular, as its formula can be applied also to actions $h_\zn$ of the monoid $(\R,\cdot)$ of multiplicative reals, so lifts of vector (and more general) bundle structures (cf. \cite{Grabowski:2009,Grabowski:2012}).

\mn Starting from local coordinates $(x^i,s)$ in $P$ associated with a local trivialization, we get the standard adapted coordinates $(x^i,s,\p_j,z)$ on $\sT^\ast P$ in which
\begin{equation}\label{e:17}
({\dts}h)_\zn(x^i,s,\p_j,z)=(x^i,\zn s,\zn\p_j,z).
\end{equation}
In other words, for the lifted action, all these coordinates are homogeneous, $x^i$ and $z$ of degree $0$ ($\Rt$-invariant), $s$ and $\p_j$ of degree 1.

The lifted action provides $\sT^\ast P$ with the structure of a principal bundle whose base is the bundle $\sJ^1L^*_P$ of first jets of sections of the line bundle $L^*_P$, dual to $L_P$ (cf. \cite{Grabowski:2013}).
{Moreover,} the lifted action $h^f$ provides $\mathsf{T}^\ast P$ with the structure of a principal bundle whose base is the bundle $\mathsf{J}^1L^\ast_P \otimes\Hom(f(L_P), L_P)$, where {$f(L_P)$} is the line bundle associated to the 1-dimensional representation of $\mathbb{R}^\times$ on $\mathbb{R}$ provided by $f$ itself.}

\mn Local coordinates $(x^i,z)$ on $L_P^*$ induce coordinates $(x^i,p_j,z)$ on $\sJ^1L_P^*$.
We therefore have the following.
\begin{proposition}\label{pjb} {Let $L$ be a line bundle and $L^*$ its dual. Then the cotangent bundle $\sT^\ast(L^*)^\ti$, equipped with the $\Rt$-action ${\dts}h$, where $h_\zn$ is the multiplication by $\zn$ in $L^*$, is an $\R^\times$-principal bundle over the manifold $\sJ^1L$ of first jets of sections of the line bundle $L$. The projection
$$\zt_1:\sT^\ast(L^*)^\ti\rightarrow\sJ^1L$$
in the adapted coordinates, takes the form
\be\label{jetproj}(x^i,s,\p_j,z)\longmapsto \big(x^i,p_j={\p_j}/s,z\big).\ee}
\end{proposition}

\section{Contact structures}\label{Sec3}
In the literature on the subject (see, e.g., \cite{Geiges:2008}), a \emph{contact structure} is a \emph{contact distribution}, i.e., a maximally non-integrable distribution $C\subset\sT M$ of corank 1 on a manifold $M$ of odd dimension $2n+1$. Traditionally, the hyperplanes forming this distribution are called \emph{contact elements}. The maximal non-integrability means that the bilinear map
$$\zW_C:C\ti_MC\to\sT M/C\,,\quad \zW_C(X,Y)=\zy([X,Y])$$
is non-degenerate. Here,
$$\zy:\sT M\to L=\sT M/C$$
is the canonical projection onto the line bundle $L=\sT M/C$, and $[X,Y]$ is the Lie bracket of the vector fields $X,Y\in C$. This shorthand notation means that these vector fields take values in $C$. It is easy to see that $\zW_C$ is well defined as a bilinear form on $C$ with values in $L$. Note that $\zy$ can be viewed as a 1-form on $M$ with values in $L$, and the two-form $\zW$ on $C$ can be viewed as `{$-\xd\zy$}'. More precisely, for any connection { $D$} on $L$, we have {$\Omega=-(\xd_D\vartheta)|_C$.}

The distribution $C$ is locally the kernel of a nonvanishing 1-form $\eta$ on $M$, and the maximal non-integrability condition is then expressed as
\begin{equation}\label{e:23}
\eta\wedge(d\eta)^n\neq 0.
\end{equation}
Such a form is called a \emph{contact form} and is not uniquely determined by $C$, since the kernels of $\eta$ and $f\eta$ are the same, provided $f$ is a nonvanishing function. From (\ref{e:23}) it is now easy to see that a 1-form $\zh$ is a contact form if and only if $f\zh$ is a contact 1-form. The contact form $f\zh$ we call \emph{(conformally) equivalent} to $\zh$. It defines the same contact distribution $C=\ker(\zh)$.
The local picture of contact forms is fully described.
\begin{theorem}[Contact Darboux Theorem] Let $\zh$ be a 1-form on a manifold $M$ of dimension $(2n+1)$. Then $\zh$ is a contact form if and only if, around every point of $M$, there are local coordinates $(z,p_i,q^i)$, $i=1,\dots,n$, in which $\zh$ reads
\be\label{Dc} \zh=\xd z-p_i\,\xd q^i.\ee
\end{theorem}
\no Any contact form $\zh$ on $M$ determines uniquely a nonvanishing vector field $\zx$ on $M$, called the \emph{Reeb vector field}, which is characterized by the equations
$$i_\zx\zh=1\quad \text{and}\quad i_\zx\xd\zh=0.$$
The Reeb vector field for the contact form (\ref{Dc}) is $\zx=\pa_z$.

\begin{proposition}
A {\it contact structure} on a manifold $M$ of odd dimension $2n+1$ is a distribution of hyperplanes $C\subset\sT M$ such that $C$ is locally a kernel of a contact form.
\end{proposition}
\no A contact structure is often understood as a manifold equipped with a global contact form $\zh$. We will call such contact structures $C=\ker(\zh)$ \emph{trivial} or \emph{cooriented}. Note that (\ref{e:23}) implies that any trivializable contact manifold must be orientable.

\noindent To work with contact structures, we shall use the language of \emph{symplectic $\Rt$-principal bundles}, due to the following observation (see  \cite{Arnold:1989,Bruce:2017,Grabowska:2022}).
\begin{proposition}\label{cosy}
Let $\zh$ be a 1-form on a manifold $M$. Then, $\zh$ is a contact form if and only if the closed 2-form
\be\label{hsf}\zw_\zh(x,s)=\xd(s\zh)(x,s)=\xd s\we\zh(x)+s\cdot\xd\zh(x)\ee
on $\cM=M\ti\R_+$ is symplectic.
\end{proposition}
\no Here, $s>0$ and $\R_+$ is viewed as the multiplicative group of positive reals. It is easy to see that $\zw_\zh$ is 1-homogeneous with respect to the $\R_+$-action $h_\zn(x,s)=(x,\zn s)$,
\be\label{homogeneity}
h_\zn^*(\zw_\zh)=\zn\cdot\zw_\zh.
\ee
The symplectic form (\ref{hsf}) is called the \emph{symplectization} of $\zh$. The homogeneity of $\zw_\zh$ can be equivalently described by
$\Ll_\n\zw_\zh=\zw_\zh$, where $\n=s\pa_s$ is the generator of the $\R_+$-action and $\Ll$ denotes the Lie derivative. Conversely, every homogeneous symplectic form $\zw$ on $\cM=M\ti\R_+$ reads as in (\ref{hsf}) for some contact form $\zh$ on $M$. Of course, we could also consider $\zw_\zh$ on $M\ti\Rt$, which seems to be superfluous, as the latter manifold has two diffeomorphic connected components. {However, if we want to glue a non-trivializable principal bundle with a homogeneous symplectic form out of trivial symplectizations, the use of $\Rt$ instead of $\R_+$ is unavoidable.}
\begin{remark}
We consider $\R_+$-bundles rather than $\R$-bundles. Of course, both pictures are equivalent, but it is more convenient to see $(\R_+,\cdot)\simeq(\R,+)$ as a subgroup in $\Rt$. Standard symplectization is often considered on the cone $\cM=M\ti\R$ instead of $\cM=M\ti\R_+$, so (\ref{hsf}) takes the form
\be\label{hsf1}\zw_\zh(x,t)=\xd(e^t\zh)(x,t)=e^t\big(\xd t\we\zh(x)+\xd\zh(x)\big).\ee
\end{remark}
For a general contact manifold $(M,C)$, we do not have a global contact form $\zh$ determining the contact distribution $C$ as its kernel. The analog of $\zh$ is $\zy:\sT M\to L=\sT M/C$, and $L^*=C^o\subset\sT^*M$, where $C^o$ denotes the annihilator of the distribution $C$. Now, it is easy to see that $C$ is a contact distribution if and only if $P=\big(C^o\big)^\ti$ is a symplectic submanifold of $\sT^*M$ with its canonical symplectic form $\zw_M$. Note that $P$ is additionally an $\Rt$-bundle with respect to the standard multiplication $h_s$ in $\sT^*M$ by non-zero reals, and its symplectic form $\zw=\zw_M\big|_P$ is 1-homogeneous, $h_s^*(\zw)=s\zw$. This is the canonical symplectization of $(M,C)$. The properties of $(P,\zw)$ lead to the following general concept (cf. \cite{Bruce:2017,Grabowski:2013}).
\begin{definition} A \emph{symplectic $\Rt$-bundle} is a principal $\Rt$-bundle $\zt:P\to M$ equipped with a 1-homogeneous symplectic form $\zw$, $h^*_\zn(\zw)=\zn\cdot\zw$ for all $\zn\in\Rt$.
\end{definition}
\no Any symplectic $\Rt$-bundle $(P,\zw)$ carries additional canonical structures: the infinitesimal generator of the $\Rt$-action, denoted with $\n$ and called the \emph{Liouville vector field}, and the nonvanishing semibasic form $\zvy=i_\n\zw$ called the \emph{Liouville 1-form}. It is a `vector potential' for $\zw$, $\xd\zvy=\zw$. The contact distribution $C$ on $M$ is the image of the distribution $\ker(\zvy)$ under the projection
$\zt:P\to M$. Note that the Liouville 1-form $\zvy$ can be viewed as a map $\Phi:P\to\sT^*M$ (cf. \cite[Theorem 2.17]{Grabowska:2023}). Indeed, $\zvy$ is semibasic, so we can put
$$\zt^*\big(\Phi(p_x))=\zvy(p_x),$$
and it is easy to see that $\Phi$ yields a canonical isomorphism of the symplectic $\Rt$-bundle $(P,\zw)$ onto $\big(L^*\big)^\ti\subset\sT^*M$ with the restriction of the canonical symplectic form $\zw_M$ on $\sT^*M$. Under this embedding the Liouville 1-form $\zvy$ is the pullback $\Phi^*(\zvy_M)$ of the canonical Liouville 1-form $\zvy_M$ on {$\sT^*M$}.

Let us stress that the Liouville 1-form is a geometric object on $P$, so the distribution $\ker(\zvy)$ does not depend on the trivialization and projects onto a contact distribution $C$ on $M$, so we call $(P,\zw)$ a \emph{symplectic cover} (symplectization) of $(M,C)$. In fact, any contact manifold $(M,C)$ admits a symplectic cover that is unique up to isomorphism.
\begin{theorem}[\cite{Grabowski:2013}]\label{main}
There is a canonical one-to-one correspondence between contact manifolds $(M,C)$ and isomorphism classes of symplectic $\Rt$-principal bundles $(P,\zw)$ over $M$, with the canonical projection $\zt:P\to M=P/\,\Rt$. The canonical representative of this class is $P=\big(L^*\big)^\ti\subset L^*$, where $L=\sT M/C$, with its canonical symplectic form.

\mn In this correspondence, the contact distribution $C$ is the projection of the kernel of the Liouville 1-form, $C=\sT\zt\big(\ker(\zvy)\big)$. {This correspondence gives rise to an equivalence of categories.}
\end{theorem}

\mn Symplectic $\Rt$-bundles are generally not trivializable. Any local trivialization induces a coordinate $s$ in fibers, and a local contact form $\zh$ on $M$ such that
\be\label{sffc}
\zw=\xd s\we\zh+s\cdot\xd\zh{.}
\ee
For coorientable contact manifolds, the \emph{Liouville vector field} is in these coordinates $\n=s\pa_s$, and the \emph{Liouville 1-form} is $\zvy=s\zh$. The contact form $\zh$ on $M$ can be obtained as the restriction of the Liouville 1-form $\zvy$ to the submanifold $\tM$ of $P$ being the locus $s=1$ which is canonically diffeomorphic with $M$ \emph{via} the projection $\zt:P\to M$. Equivalently, this locus is the image of a (local) section $\za:M\to P$ of $P$, $\tM=\za(M)$. Since we can view $\za$ as a 1-form $\wt{\za}=\Phi\circ\za:M\to\sT^*M$ on $M$,
\be\label{cfnrs}\zh_\za\simeq\zvy\,\big|_{\za(M)}=\wt{\za}^*(\zvy_M)=\wt{\za}.\ee
The latter identity follows from the universal property of the Liouville 1-form $\zvy$: for any 1-form $\zb:M\to\sT^*M$ on $M$, we have $\zb=\zb^*(\zvy_M)$. Note also that the 1-form $\wt{\za}$ can be viewed as the map
$$\wt{\za}=\za\circ\zy:\sT M\to\R,$$
where we understand $\za$ as a linear function $\zi_\za$ on $L$. This way, we get the following proposition, which we will use for local descriptions in the general case.
\begin{proposition}\label{equiv}
Let $(P,\zw)$ be a symplectic cover of a coorientable contact manifold $(M,C)$ with the projection $\zt:P\to M$. Then, there are canonical one-to-one correspondences between
\begin{description}
\item{(a)} sections $\za:M\to P$;
\item{(b)} contact forms $\zh_\za$ on $M$ representing $C$, $C=\ker{\zh_\za}$, determined by the condition
$\zt^*(\zh_\za)=\zvy$ on the submanifold $\za(M)\subset P$;
\item{(c)} 1-homogeneous functions $s:P\to\Rt$, given by $s\circ\za=1$;
\item{(d)} regular linear functions $\zi_\za:L\to\R$ on $L=\sT M/C$;
\item{(e)} regular linear functions $\wt{\za}:\sT M\to \R$ vanishing on $C$, given by $\wt{\za}=\zi_\za\circ\zy$.
\end{description}
{Here, the regularity means that the vertical derivative is nonvanishing, i.e., the linear function corresponds to a nonvanishing section of the dual vector bundle.}
\end{proposition}
\begin{remark}
The above interpretation of contact structures as symplectic $\Rt$-bundles $(P,\zw)$ is very useful in contact Hamiltonian mechanics \cite{Grabowska:2022,Grabowska:2023a}, simplifying the traditional approaches using contact forms. We simply define \emph{contact Hamiltonians} as 1-homogeneous functions $H$ on $P$. Then, the corresponding Hamiltonian vector field $X_H$ on $P$ is $\Rt$-invariant, thus projects onto a contact vector field on the base contact manifold $M$. Note that, alternatively, we can view 1-homogeneous functions on $P$ as linear functions on $L^*$, thus as sections of the line bundle $L$. Also, reductions of contact Hamiltonian systems can be carried out in this framework \cite{Grabowska:2023}. All this can be easily extended to Jacobi structures (known also as \emph{local Lie algebras} \cite{Kirillov:1976} or \emph{Jacobi bundles} \cite{Marle:1991}), understood as Poisson $\Rt$-bundles \cite{Bruce:2017,Grabowski:2013}. Note also that it is often much easier to interpret contact manifolds as symplectic $\Rt$-bundles than to try to define the contact structure directly.
\end{remark}
\begin{example}
For a manifold $M$, the cotangent bundle  $\sT^*M$ with the zero section removed, $(\sT^*M)^\ti$, is an $\Rt$-bundle with respect to the multiplication by reals in $\sT^*M$. The canonical symplectic form $\zw_M$ restricted to $(\sT^*M)^\ti$ is still symplectic and 1-homogeneous, so we deal with a symplectic $\Rt$-bundle. According to Theorem \ref{main}, this defines a canonical contact structure on the projectivized cotangent bundle
$\Pe\,\sT^*M=(\sT^*M)^\ti/\Rt$. This structure is {never coorientable for $\dim(M)>1$}. If we quotient $(\sT^*M)^\ti$ by $\R_+$ instead of $\Rt$, we get a contact structure on a bundle of spheres $(\sT^*M)^\ti/\R_+$ over $M$, this time coorientable, but generally not canonically trivial.
\end{example}

\section{Canonical contact structures on jet bundles}
{Let $L\to M$ be a line bundle, $L^*$ be its dual, and $(L^*)^\ti$ be the corresponding $\Rt$-subbundle with coordinates $(x^i,s)$ and the standard $\Rt$-action, $h_\zn(x,s)=(x,\zn s)$. As we already know (see (\ref{jetproj})), $\sT^*(L^*)^\ti$ is canonically an $\Rt$-bundle over $\sJ^1L$ with the lifted $\Rt$-action $\dts h$. The following is well known (see e.g. \cite{Grabowski:2013}).
\begin{proposition}\label{triv}
For every line bundle $\zt:L\to M$, there is a canonical contact structure on the jet bundle $\sJ^1L$ for which the symplectic $\Rt$-bundle $\sT^*(L^*)^\ti$ is a symplectic cover. This contact structure is trivializable if and only if the line bundle $L\to M$ is trivializable.
\end{proposition}
\begin{remark} The jet bundle $\zt^1:\sJ^1L\to M$ is a vector bundle, and it is easy to see that the contact structure on the vector bundle $E=\sJ^1L$ is \emph{linear}, i.e., it is locally induced by linear contact forms. Equivalently, in a more advanced geometrical language, the contact distribution $C\subset \sT E$ is a double vector subbundle in the \emph{double vector bundle} $\sT E$ (cf. \cite{Grabowski:2009,Grabowski:2012}), i.e., it is also a vector subbundle with respect to the projection $\sT\zp^1:\sT E\to\sT M$. One can show \cite{Grabowski:2013} that all linear contact structures are of this type, like all linear symplectic structures are isomorphic to cotangent bundles with their canonical symplectic forms.
\end{remark}}
\begin{example}\label{jetm} Let us consider again the M\"obius band $B\to S^1$ from Example \ref{MB}.
The line bundle structure on $B^*\to S^1$ has a dual description by two charts, $\cO^*$ and $\cU^*$, completely analogous to that in (Example \ref{MB}) for $B$. As the domains of two charts in $\sJ^1B^*$, we take
$\bar{\mathcal{O}}=(\zb)^{-1}(\mathcal{O}^*)$  and  $\bar{\mathcal{U}}=(\zb)^{-1}(\mathcal{U}^*)$, where $\zb:\sJ^1B^*\to B^*$ is the canonical projection.
The adapted coordinates in $\bar{\mathcal{O}}$ are
$$(x, p, z)\in\,]0,1[\times\R\times \R,$$
while the adapted coordinates in $\bar{\mathcal{U}}$ are
\[(x',p',z')\in\,\Big]1/2,3/2\Big[\times\R\times \R,\]
with the transition map
\[(x',p',z')=\bar\zF(x,p,z)=
\begin{cases}
(x,p,z)\quad\text{if}\quad x\in]1/2,1[\\
(x+1,-p,-z)\quad\text{if}\quad x\in]0,1/2[\,.
\end{cases}
\]
The coordinate $p$ changes sign in the same way as $z$, because if a section $\sigma$ is given in {the} chart $\mathcal{O^*}$ as $x\mapsto (x,z(x))$, then $p(\sj^1\sigma(x))=\frac{\pa z}{\pa x}$.

\mn {Out of the Liouville 1-form on $\sT^*B$ we get contact forms of these two trivializations: $\zh=\xd z-p\xd x$ and $\zh'=\xd z'-p'\xd x'$. The contact form $\zh$ in $\wt{\cO}$, changes under the transition map into $\zh'$ for $x\in]1/2,1[$, and into $-\zh'$ for $x\in]0,1/2[$.} Since the 1-forms $\zh$ and $\zh'$ are basic, we can view them as contact forms on the charts $\bar\cO$ and $\bar\cU$ in $\sJ^1B^*$. Of course, $\ker(\zh)$ and $\ker(\zh')$ agree on the intersection of charts and define the canonical contact distribution $C$ on $\sJ^1B^*$. In other words, we can view the contact structure on $\sJ^1B^*$ as represented by the contact form $\zh$ on $\bar{\cO}$, and $\zh'$ on $\bar{\cU}$.

Interestingly, the non-coorientable contact structure lives on the vector bundle $\sJ^1 B^*\rightarrow S^1$, which is trivializable. Let us consider the following two sections of the bundle $B^*\rightarrow S^1$, written in coordinates from the chart $\mathcal{O}^*$ as
$$\sigma_1(x)=(x,\sin(\zp x)), \qquad \sigma_2(x)=(x, \cos(\zp x)).$$
It is easy to check that they are well defined globally, since
$$\sin\big(\zp(x+1)\big)=-\sin(\zp x)\quad\text{and}\quad \cos\big(\zp(x+1)\big)=-\cos(\zp x).$$
Now, in the chart $\bar{\mathcal{O}}$, we can write the first jet prolongations of $\sigma_1$ and $\sigma_2$ as
$$\sj^1\sigma_1(x)=(x,\sin(\zp x),\zp\cos(\zp x)), \qquad \sj^1\sigma_2(x)=(x,\cos(\zp x),-{\zp\sin(\zp x)}).$$
The two sections $\sj^1\sigma_1$ and $\sj^1\sigma_2$ are global, nonvanishing, and linearly independent sections of $\sJ^1B^*\rightarrow S^1$. This shows that the vector bundle $\sJ^1L\to M$ may be trivializable even for a non-trivializable line bundle $L\to M$.
\end{example}

\section{Calibrations and paired structures}\label{sec-cal}
Any $\Rt$-bundle $\zt:P\to M$ is of the form $P=\Lt$ for a line bundle $\zt:L\to M$. Let us choose a VB-metric $g_L$ on $L$. Recall that a \emph{VB-metric} on a vector bundle $E$ is a symmetric form $g_E\in\Sec(\sS^2E^*)$ that induces scalar products on fibers. The metric $g_L$ is completely determined by the norm,
\be\label{norm}{\s:L\to[0,+\infty[}, \quad \s(v)=\nm{v}=\sqrt{g_L(v,v)}.\ee
The function $\s$ is positive and positively homogeneous on $P=\Lt$, i.e., {$\s(tv)=|t|\s(v)$ for $t\in\Rt$}. The latter simply means that $s$ is $f$-homogeneous where $f:\mathbb{R}^\times\rightarrow \mathbb{R}^\times$ is the absolute value $s\mapsto|s|$. Conversely, any positive and positively homogeneous function $\s$ on $P=\Lt$ defines a VB-metric $g_L$ \emph{via} (\ref{norm}).
Such functions on an $\mathbb{R}^\times$-bundle $P$ will be called \emph{calibrations}, and $(P,\s)$ -- a \emph{calibrated $\Rt$-bundle}. The terminology is motivated by the fact that fixing a calibration fixes each local fiber coordinate up to a factor, so it plays the role of reference data.

It is well known that VB-metrics always exist, so calibrations do exist on every $\Rt$-bundle. They differ by a factor being a (pullback of) a positive function on $M$. Since in the case of non-trivializable $P$ we do not have nonvanishing sections, {these are calibrations which we will use instead}.

\mn It is obvious that any calibration is a regular function, $\xd\s\ne 0$. This immediately implies the following.
\begin{theorem}\label{calibration}
Every calibration $\s$ on an $\Rt$-bundle $\zt:P\to M$ defines a horizontal foliation $\cF_\s$ on $P$ whose leaves $\tM_c$ are the level sets $\s=c>0$. The leaves of $\cF_\s$ are 2-sheet covers of $M$ under the projection $\zt_c=\zt\big|_{\tM_c}$. {For connected $M$,} they are connected if and only if $P$ is not trivializable; otherwise, they consist of two components. The pullback bundle $\tP_c=\zt^*_cP$ of $P$ over $\tM_c$ is a trivial $\Rt$-bundle, with the tautological global section $\zs_c:\tM_c\to\tP_c$: if $y\in\tM_c$, then $\big(\tP_c\big)_y=P_{\zt(y)}$ and $\zs_c(y)=y\in P_{\zt(y)}$.
\end{theorem}
\begin{remark}\label{conn} The foliation $\cF_\s$ is $\Rt$-invariant and can be viewed as a flat principal connection on $P$ called an \emph{$\s$-connection}. The corresponding $\Rt$-invariant horizontal distribution $\cH_\s$ consists of vectors tangent to the leaves of the foliation. It is the kernel of the connection 1-form $\zz^\s=\xd\s/\s$ on $P$. The connection 1-form is a true 1-form, since the Lie algebra of $\Rt$ is $\R$, with the canonical fundamental vector field (more precisely, the negative of the fundamental vector field) $\n$ on $P$.
The horizontal lifts of vector fields $X$ on $M$ we will denote $\hX$, and the pullbacks to $P$ of symmetric or skew-symmetric differential forms $\zb$ on $M$ with $\wh{\zb}$. Given a calibration $\s$, in a neighbourhood $U$ of each point of $M$ we can find a local trivialization $\zt^{-1}(U)\simeq U\ti\Rt$ such that the fiber coordinate $s$ satisfies $\s=|s|$.  Since, in such coordinates, we have
$$\big(f^a(x)\pa_{x^a}\big)^\s=f^a(x)\pa_{x^a}\quad\text{and}\quad\big(h_a(x)\xd x^a\big)^{\wh{}}=h_a(x)\xd x^a,$$
we will sometimes identify, with some abuse of notation, vector fields $X$ on $M$ with their horizontal lifts in $P$, and differential forms on $M$ with their pullbacks to $P$. This should be clear from the context.
\end{remark}
\no Fixing a calibration $\s$, we easily get the following (see, e.g., \cite{Grabowski:2012a,Grabowski:2006}).

\begin{proposition}\label{phf}
On every $\Rt$-bundle $\zt:P\to M$ with a fixed calibration $\s$, there is a unique maximal atlas of local trivializations of $P$, consisting of an open cover $\{U_\za\}_{\za\in\zL}$ of $M$, and local trivializations
\be\label{gtriv}\zf_\za:\zt^{-1}(U_\za)\to U_\za\ti\Rt\ee
such that the fiber coordinate $s_\za$ associated with $\zf_\za$ satisfies $|s_\za|=\s\big|_{U_\za}$. For this atlas, the transition maps are reduced to a sign change in the fiber coordinates,
$$\zf_{\za\zb}:(U_\za\cap U_\zb)\ti\Rt\to (U_\za\cap U_\zb)\ti\Rt\,,\quad
\zf_{\za\zb}(x,s)=(x,\pm s)\,.$$
Consequently, in the presence of a calibration, one can reduce locally problems concerning homogeneous structures on $P$ to the case of the trivial $\R_+$-bundles, reducing local trivializations (\ref{gtriv}) to $U_\za\ti\R_+$.
\end{proposition}
\begin{definition}\label{def-cal} The atlas of local trivializations of $P$, described in Proposition \ref{phf}, we call the \emph{$\s$-atlas} and the fiber coordinates $s$ of the corresponding local trivializations (i.e., satisfying $|s|=\s$) \emph{$\s$-normal}.
\end{definition}
\begin{corollary}\label{cr1}
Let $(M,C)$ be a contact manifold, and let $(P,\zw)$ be its symplectic cover. If $\s$ is a calibration on $P$, then the pullback contact structure $p^*C$ on the 2-sheet cover $p:\tM=\tM_1\to M$ admits a canonical global contact form $\tzh$ which locally, on the {two sheets} covering $U_\za$, is $\pm\zh_\za$. In other words, we can find an open covering  $\{U_\za\}$ of $M$ and local contact forms $\zh_\za$ inducing $C$ on $U_\za$ such that $\zh_\za=\pm \zh_\zb$ on $U_{\za\zb}=U_\za\cap U_\zb$.
The contact form $\tzh$ is the restriction to $\tM$ of the Liouville form $\zvy$ on $P$.
\end{corollary}
\begin{remark}\label{paired}
We can also view $\tzh$ as a horizontal submanifold of $\sT^*M$ being a 2-sheet cover of $M$, In other words, $\tzh$ is a section of the fiber bundle
$$|\sT^*M|=\big(\sT^*M\big)^\ti/\Z_2\to M.$$
Such sections we will call \emph{paired sections} of $\sT^*M$; they are nonvanishing by definition. Similarly, the Reeb vector field $\tzx$ of $\tzh$ can be viewed as a paired vector field $|\zx|$ on $M$. Of course, this definition makes sense for any vector bundle $E\to M$ or, more generally, for any fiber bundle with a canonical action of $\Z_2$ in fibers. In \cite{Grabowski:2012a,Grabowski:2006} it was applied to the case of line bundles to show that paired sections exist for all line bundles.

\mn More generally, we can consider \emph{paired structures} on a manifold $M$ which are locally represented by pairs of structures  up to a sign. More precisely, if a certain geometric structure is locally represented by a system of nonvanishing tensor fields $T_\za=(T_\za^1,\dots,T_\za^k)$ satisfying certain compatibility conditions which are invariant with respect to the sign change, then the family of local pairs $\{T_\za,-T_\za\}$ which coincide on the intersections of charts defines a paired structure. Such a paired structure we will denote $|T|=|T^1,\dots,T^k|$. This is generally different from $(|T^1|,\dots,|T^k|)$. It is clear that any paired structure $|T|$ on $M$ corresponds to a standard structure $T$ on a 2-sheet cover $\tM$ of $M$.
\end{remark}

\begin{example}\label{epaired}
Let $\zt:E\to M$ be a vector bundle, and let $\GL(E)\to M$ be the bundle of linear automorphisms of fibers, i.e., the subbundle in $\Hom(E;E)=E^*\ot_M E\to M$ of invertible elements. It is a bundle of Lie groups. Let $w(x)\in\R[x]$ be a polynomial of the form $w(x)=1+a_1x^1+\cdots+a_nx^{n}$. A \emph{$w$-structure} on $E$ is a section $\phi:M\to\GL(E)$ such that $w\circ\phi=0$. If $w$ is an even function, $w(x)=w(-x)$, then a \emph{paired $w$-structure} is a section $|\phi|:M\to|GL(E)|$ such that $w\circ|\phi|=0$, where $|\GL(E)|\to M$ is the quotient bundle $\GL(E)/\Z_2$, with $\Z_2$ interpreted as the normal subgroup $\{\pm\,\id\}$ in each fiber.

In this way, we can obtain a \emph{paired complex structure on $E$}, with $w(x)=1+x^2$, or a \emph{paired product structure} with $w(x)=1-x^2$. The definition makes sense, as any paired $w$-structure is locally represented by a pair $\pm\,\phi$ of $w$-structures on the vector bundle $E$ and $w$ is insensitive to the change of sign in the argument. Note that paired complex structures can live on a wider class of vector bundles than complex structures; {the bundles may be non-orientable if rank of $E$ is odd, i.e., ${\cm\we^{top}}E\to M$ may be a nontrivial line bundle. Complex structures live on orientable vector bundles only.}
\end{example}
\section{Levi and Sasaki structures}
Let $M$ be a cooriented contact manifold with a contact form $\eta$, and let $\zx$ be the corresponding Reeb vector field,
\begin{align*}
    {i_\zx}\eta=1,\,\,  \text{and} \,\,\, {i_\zx} \xd\eta=0.
\end{align*}
Consequently, the tangent bundle $\sT M$ has the decomposition
\begin{align*}
    \sT M=\langle\zx\rangle \oplus C,
\end{align*}
where $C=\ker(\zh)$ is the \emph{contact distribution} and $\langle\zx\rangle$ is the line subbundle generated by $\zx$.
\begin{definition}\label{g-as} A Riemannian metric $g$ on $(M,\zh)$ is an \emph{associated metric} for the contact form $\zh$ if, {for all vector fields $X,Y$ on $M$},
\[
    \eta(X)=g(X,\zx),
\]
and there exists an endomorphism $\phi$ of $\sT M$ satisfying $\phi^2=-\id_{\sT M} + \eta\otimes\zx$ and
{$$g\big(X,Y\big)=\xd\eta\big(X,\phi(Y)\big).$$}
We refer to $(\phi,\zx,\zh,g)$ as a \emph{contact metric structure} and to $M$ with such a structure as a \emph{contact metric manifold} {\cite{Blair:1976,Sasaki:1961}}.
\end{definition}
\no In particular, the contact subbundle is orthogonal to the Reeb vector field, $\phi(\zx)=0$, and $\phi(C)=C$, so $\phi$ can be seen as an endomorphism $\pc$ of {$C=\ker(\zh)$}, satisfying $\pc^2=-\id_C$, and extended trivially to the whole $\sT M=\la \zx\ran\op C$. We will often use this point of view.
\begin{remark}
Note that this definition strongly depends on the choice of $\zh$ in the conformal class of $\zh$, so it has no clear meaning in the contact distributional setting.
\end{remark}
\no {Besides contact and contact metric structures, the literature contains many such `almost'
structures, which, in our opinion, cause some confusion, as the logic of the terminology is not
respected.}
\begin{definition}
Let $M$ be an odd-dimensional manifold. An \emph{almost contact structure} on $M$ consists of the following:
\begin{enumerate}
    \item a $1$-form $\eta$,
    \item a vector field $\xi$,
    \item an endomorphism $\phi:\sT M\to\sT M$,
\end{enumerate}
such that they satisfy $\iota_\xi\eta=1$ and
$$\phi^2=-\id_{\sT M} + \eta\otimes\zx.$$
A manifold $M^{2n+1}$ endowed with an almost contact structure $(\phi, \xi, \eta)$ is called an \emph{almost contact manifold}.
\end{definition}
\begin{remark}
{The phrase `almost A'} usually suggests that the object in question is `not quite A' and, logically, `A' is always `almost A'. The concept of an almost contact structure does not satisfy this logic rule: a contact manifold is generally not almost contact, as there is no canonical $\phi$ associated with a contact structure. On the other hand, such $\phi$ that makes a {cooriented} contact manifold into an almost contact one always exists. In any case, the name `almost contact' has historical origins and is commonly used in the above sense, so we will respect this convention.
\end{remark}
\no Before we discuss the traditional approaches to the concept of a Sasakian manifold, we briefly recall the relation between contact geometry and CR geometry; for more details, see \cite{Blair:2010, Boyer:2008, Ianus:1972, Ornea:2007, Ornea:2024}.

\mn Let us recall that a $(1,1)$ tensor $J\in E^*\ot_ME$ on a vector bundle $\zp:E\to M$ is called a \emph{complex structure} on the vector bundle $E$ if $J^2=-\id_E$. A complex structure on the tangent bundle $\sT M$ is called an \emph{almost complex structure on $M$}, and it is called a \emph{complex structure on $M$} if it is \emph{integrable}.

The integrability of $J$ can be characterized by the celebrated result of Newlander and Nirenberg \cite{Newlander:1957} as the vanishing of the \emph{Nijenhuis torsion} $N_J$ of $J$,
\be\label{Nt0}N_J(X,Y)=[JX,JY]-J\big([JX,Y]+[X,JY]-J[X,Y]\big),\ee
which for $J^2=-\id$ takes the form
\be\label{Nt}\cN_J(X,Y)=\big([JX,JY]-[X,Y]\big)-J\big([JX,Y]+[X,JY]\big),\ee
{but we use the above tensor for arbitrary $J$, which differs then from  $N_J$. It was used
in the concept of a \emph{CR structure}, introduced in 1968 by Greenfield \cite{Greenfield:1968}, which was an attempt to develop a complex-like geometry for general distributions $C$ replacing $\sT M$.}

\begin{definition}\label{crdef}
If $M$ is a smooth manifold and $\mathcal{H}$ is a complex subbundle of the complexified tangent bundle $\sT^\C M$ such that $\mathcal{H}\cap\overline{\mathcal{H}}={0}$, then the pair $(M,\mathcal{H})$ is called an \emph{almost CR manifold} if
{$$[\Sec(\cH),\Sec(\cH)]\subset\Sec(\cH\op\overline{\mathcal{H}}),$$}
and \emph{integrable} (or a \emph{CR manifold}) if $\mathcal{H}$ is involutive.
\end{definition}
\no Alternatively, we can start with a subbundle $C$ of the tangent bundle $\sT M$, together with a complex  structure $\phi_C$  on $C$, $\phi_C^2=-\id_C$. Then, $(C, \phi_C)$ is an \emph{almost CR structure} if
\be\label{acr}[\pc(X),\pc(Y)]-[X,Y]\in C,\ee
for all $X,Y\in C$. Since $\pc^2=-\id_C$, it is easy to see that (\ref{acr}) is equivalent to
\be\label{acr1}[\pc(X),Y]+[X,\pc(Y)]\in C.\ee
{Under this condition, the expression
\be\label{Nt1}\cN_\pc(X,Y)=\big([\pc(X),\pc(Y)Y]-[X,Y]\big)-\pc\big([\pa(X),Y]+[X,\pc(Y)]\big),\ee
where $X,Y\in C$, takes values in $C$, and defines a tensor $\cN_\pc\in \Sec(\we^2C^*\ot C)$.}

An almost CR structure is called \emph{integrable} (or a \emph{CR structure}) if the tensor $\cN_\pc$ vanishes. This is the definition we will use in the sequel.
Note that some authors understand almost CR structures $(C,\pc)$ as simply complex structures on the vector bundle $C$, $\pc^2=-\id_C$.

The relation to Definition \ref{crdef} is clear: one can define
{\cm$$\mathcal{H}=\{X-i\pc(X)\,|\,X\in C\},$$}
and it is easy to show that $\mathcal{H}\cap\overline{\mathcal{H}}={0}$, and that $\mathcal{H}$ is involutive if and only if $\cN_\pc=0$. Conversely, viewing $\sT M$ as canonically embedded into $\sT^\C M$, we have
{$$C=(\cH\op\overline{\mathcal{H}})\cap\sT M$$}
and $\cH$ (or $\overline{\mathcal{H}}$) equals $\{X-i\pc(X)\,|\,X\in C\}$.
The case $\dim(M)=2n+1$ with a subbundle $C$ of rank $2n$ is particularly interesting to us, as it is related to contact distributions. Starting from an almost contact structure $(\phi,\xi,\eta)$ on $M$ and restricting $\phi$ to the subbundle $C$, one obtains a complex structure $\pc=\phi|_C$ on $C$.
\begin{remark}\label{pcr}
Instead of almost CR structures on $C$ we can consider \emph{paired almost CR structures} (cf. Example \ref{epaired}) represented locally by pairs of almost CR structures $\{\pc,-\pc\}$. This concept is correct, as $\big(\pm\pc\big)^2=-\id_C$ and the condition (\ref{acr}) is insensitive to the sign before $\pc$. In the complex setting, we have a pair of complementary subbundles locally, $\{\cH,\ol{\cH}\}$ of $C^\C=C\ot\C\subset\sT^\C M$,
{$$C\ot\C=\cH\op\ol{\cH}.$$}
We want to make clear that this pair is unordered, $\ol{\ol{\cH}}=\cH$. The relation to the previous setting is by putting
$$\cH=\{X\pm i\pc(X)\,|\,X\in C\}.$$
Also, the concept of integrability works for paired CR structures, since (\ref{Nt0}) and (\ref{Nt}) are sign-insensitive.

This `paired' concept is related to the fact that $i=\sqrt{-1}$ is by no means uniquely determined, as we are unable to distinguish between two square roots of $-1$ in the field of complex numbers. Consequently, holomorphic and anti-holomorphic functions in complex analysis are exchangeable. The distinction we use comes from the model $\C=\R\op i\R$. Paired complex structures on vector bundles (maybe \emph{local complex structures} would be a better name) form a weaker concept (complex analysis makes sense only locally), but they can work for non-orientable vector bundles { (cf. Example \ref{epaired})}.
\end{remark}

\mn Let $\pc$ be a complex structure on a corank $1$ distribution $C$ and $C=\ker(\eta)$, where $\eta$ is a nowhere-vanishing $1$-form. {Note that this condition implies that the line bundle $\sT M/C$ is trivializable.} The \emph{Levi form} $g_C$ is defined by
\be\label{Levi}g_C(X,Y)=\xd\eta\big(X,\pc(Y)\big),\ee
and if it is nondegenerate, then $\eta$ is a contact form.
\begin{proposition}\label{pin} The following are equivalent:
\begin{description}
\item{(a)} The 2-form $\xd\zh$ is $\pc$-invariant,
$\xd\zh\big(\pc(X),\pc(Y)\big)=\xd\zh(X,Y)$;
\item{(b)} The Levi form $g_C$ is symmetric, $g_C(X,Y)=g_C(Y,X)$;
\item{(c)} The Levi form $g_C$ is $\pc$-invariant,
$g_C\big(\pc(X),\pc(Y)\big)=g_C(X,Y)$;
\item{(d)} $(C,\pc)$ is an almost CR structure.
\end{description}
\end{proposition}
\begin{proof}
(a) $\Rightarrow$ (b)  We have
$$g_C(X,Y)=\xd\zh\big(X,\pc(Y)\big)=\xd\zh\big(-\pc(X),Y\big)
=\xd\zh\big(Y,\pc(X)\big)=g_C(Y,X).$$

 (b) $\Rightarrow$ (c) follows from
\beas &g_C\big(\pc(X),\pc(Y)\big)=g_C\big(\pc(Y),\pc(X)\big)=\xd\zh\big(\pc(Y),-X\big)\\
&=\xd\zh\big(X,\pc(Y)\big)
={g_C(X,Y)}.\eeas

 (c) $\Rightarrow$ (d) We have
\beas &\zh\big([X,\pc(Y)]\big)=-\xd\zh\big(X,\pc(Y)\big)=-{g_C\big(X,Y\big)}\\
&=-g_C\big(\pc(X),\pc(Y)\big)=\xd\zh{\big(\pc(X),Y\big)}=-\zh\big([\pc(X),Y]\big),
\eeas
so
$$\zh\big([X,\pc(Y)]+[\pc(X),Y]\big)=0$$
and
$$[X,\pc(Y)]+[\pc(X),Y]\in C.$$

\mn The equivalence (d) $\Leftrightarrow$ (a) follows from
$$\zh\big([\pc(X),\pc(Y)]-[X,Y]\big)=\xd\zh(X,Y)-\xd\zh\big(\pc(X),\pc(Y)\big)=0.$$
\end{proof}
\no A complex structure $\pc$ on $C=\ker(\zh)$ is called \emph{strictly pseudoconvex} if the Levi form is symmetric {and either positive or negative definite}. Proposition \ref{pin} immediately implies that in this case $(C,\pc)$ is an almost CR structure. The strict pseudoconvexity does not depend on the choice of $\zh$ in its conformal class; however, the Levi form does. To simplify the terminology, we propose the following definition.
\begin{definition}\label{Ls}
A \emph{Levi structure} on a cooriented contact manifold $(M,\zh)$ is a complex structure $\pc$ on the contact distribution $C=\ker(\zh)$ such that the Levi form (\ref{Levi}) is symmetric and positive definite. If $\pc$ is additionally a CR structure ($\pc$ is integrable), we will speak about an \emph{integrable Levi structure}. The structures $(M,\zh,\pc)$ we then call \emph{(integrable) Levi manifolds}.
\end{definition}
\begin{remark}\label{rmain}
Our integrable Levi structures are sometimes called in the literature \emph{strictly pseudoconvex CR structures}. This terminology comes from the complex geometry. We want to {emphasise} the Levi structure, since it is fundamental in our framework.

\mn The endomorphism $\pc$ acts on $C$, but in the presence of $\zh$ we have the canonical splitting $\sT M=\la\zx\ran\op C$, where $\zx$ is the Reeb vector field. As $\zx$ is a contact vector field, $[\zx,C]\subset C$. The canonical associated metric is then
\be\label{Lmetric} g_M=\zh\ot\zh+\gc,\ee
where $\gc$ is the unique trivial extension of $g_C$ to a symmetric tensor on $M$ satisfying $i_\zx\gc=0$. These are exactly the metrics considered in the traditional approaches to Sasaki manifolds. This immediately implies that our Levi structures are nothing but contact metric structures from Definition \ref{g-as}. As we are working exclusively with contact structures, in our approach we will ignore almost contact structures and use the simplified notation $(\zh,\pc)$, instead of $(\phi,\zx,\zh,g)$ appearing in Definition \ref{g-as}, as the superfluous parts $\zx,g$ are completely determined by $\zh$ and $\pc$. We will later work with slightly more general metrics on contact manifolds.
\end{remark}
\begin{proposition}\label{Killing} Let $(M,\zh,\pc)$ be a Levi structure, and $\bp:\sT M\to\sT M$ be the extension of $\pc$ satisfying $\bp(\zx)=0$. Then,
$\cN_\bp=0$ if and only if $\pc$ is integrable and, additionally,
\be\label{Kcond}
[\zx,\pc(X)]=\pc\big([\zx,X]\big)
\ee
for every vector field $X\in C$. The condition (\ref{Kcond}) can be rewritten as $\Ll_\zx\bp=0$ and is equivalent to the fact that {$\zx$} is a Killing vector field for the Levi metric (\ref{Lmetric}).
\end{proposition}
\begin{proof}
Obviously, $\cN_\bp=0$ implies immediately $\cN_\pc=0$. Since $\bp(\zx)=0$, we get from (\ref{Nt})
$$\cN_\bp(X,\zx)=[\zx,X]+\bp\big([\zx,\bp(X)]\big)=0$$
for all vector fields $X\in C$, which is another form of (\ref{Kcond}). The converse is obvious. We have
$$\bp\big([\zx,X]\big)-[\zx,\bp(X)]=\bp\big(\Ll_\zx X\big)-\Ll_\zx\big(\bp(X)\big)=-\big(\Ll_\zx\bp\big)(X),
$$
so $\Ll_\zx\bp=0$. As the Reeb vector field respects $\zh$, we have $\Ll_\zx\zh^2=0$, so it remains to show the equivalence of (\ref{Kcond}) with $\Ll_\zx g_C=0$. We have
\begin{align*}&\Ll_\zx g_C=\Ll_\zx\big(\xd\zh\circ(\id\ot\bp)\Big)
=(\Ll_\zx\xd\zh)\circ\big(\id\ot\bp\big)+\xd\zh\circ(\id\ot\Ll_\zx\bp)\\
&={\xd\zh\circ(\id\ot\Ll_\zx\bp)},
\end{align*}
and our statement follows, as $\xd\zh$ is nondegenerate on $C$.
\end{proof}
\no Note that $\zx$ is a contact vector field, $[\zx,C]\subset C$, so it makes also sense to write $\Ll_\zx\pc$ instead of $\Ll_\zx\bp$, etc.
\begin{definition}
We call a Levi structure $(M,\zh,\pc)$ a \emph{Sasakian structure} if {$\cN_\bp=0$}.
Paired Sasakian structures are defined in an obvious way.
\end{definition}
\no Proposition \ref{Killing} implies that, in the case of Sasakian structures, the Reeb vector field is a Killing vector field for the Levi metric $g_M$ on $M$.

\mn One could ask what the Nijenhuis torsion $N_\bp$ is. It is easy to see that, for $X,Y\in C$,
\be\label{Ncc} N_\bp(X,Y)-{ \cN_\bp(X,Y)}=\big(\id+\bp^2\big)\big([X,Y]\big)=\zh\big([X,Y]\big)\zx=-\xd\zh(X,Y)\zx.\ee
This way, we end up with a traditional definition of a Sasaki manifold.

\begin{corollary}\label{frueN}
A Levi structure $(M,\zh,\pc)$ is Sasakian if and only if, for {all} vector fields $X,Y\in C$,
\be\label{Nij}
N_\bp(X,Y)+\xd\zh(X,Y)\zx=0.
\ee
\end{corollary}
\begin{remark}
In 1961, Sasaki \cite{Sasaki:1961} defined what we just called a \emph{Sasakian structure}. It is common in mathematics that, {after a concept has been studied for a while}, the original definition is transformed into a simpler and more appealing form. Since Sasaki's initial paper, many discovered properties have been used to rephrase the definition of a Sasakian manifold, and we shall mention a few in this article.
\end{remark}
\no Let $(M,\zh,\pc)$ be a Levi manifold. Any vector field on the product manifold $\cM=M\ti\R$ can be written as $(X,f\pa_t)$, where $t$ is the coordinate on $\R$, $X$ is a vector field tangent to the foliation $t=const$, and $f$ is a smooth function on $\cM$. Define a tensor $J$ on $\cM$ by
\be\label{nJ}
   J\Big(X, f\pa_t\Big)=\Big(\bp(X)-f\xi, \eta(X)\pa_t\Big).
\ee
It can be shown that the tensor $J$ is an almost complex structure on $\cM$. If $J$ is integrable, the Levi structure is called \emph{normal}. We have the following (cf. \cite{Ianus:1972}).
\begin{theorem}\label{Ianus}
A Levi structure $(M,\zh,\pc)$ is normal if and only if it is integrable (Sasakian).
\end{theorem}
In the computation of $N_J$ in \cite{Blair:2010, Boyer:2008}, the authors consider four tensors, often denoted by $N^1, N^2, N^3,$ and $N^4$, given by
\begin{align}\label{4ten}
    &N^1(X,Y)= N_J(X,Y)=N_{\bp}(X,Y)+\xd\zh(X,Y)\xi,\\\nonumber
    &N^2(X,Y)= (\Ll_{\bp(X)}\eta)(Y)-(\Ll_{\bp(Y)}\eta)(X),\\\nonumber
    &N^3(X)= (\Ll_{\xi}\bp)(X),\\\nonumber
    &N^4(X)= (\Ll_{\xi}\eta)(X),
\end{align}
where $X,Y$ are vector fields on $M$. If $N^1$ vanishes, then $J$ is integrable, but it is important to notice that in this case also $N^2, N^3$, and $N^4$ vanish. A contact structure $(M,\zh)$ for which the tensors $N^2,N^3,N^4$ vanish is called in the literature a \emph{K-contact structure} (cf. \cite{Cappelletti:2015}), a concept {we do not need here.} The condition $N^1=0$ is exactly our condition telling that $(\zh,\pc)$ is a Sasakian structure (Corollary \ref{frueN}).
\begin{example}\label{sasa} Consider $M=\R^{2n+1}$, together with global coordinates $(x^1,\cdots,x^n,y^1,\cdots,y^n,z)$ and the Darboux contact form
\begin{align*}
    \eta=\xd z-\sum^n_{i=1}y^i\xd x^i.
\end{align*}
One can show that the Riemannian metric
$$g=\eta\otimes\eta+\sum^n_{i=1}{\big(\xd x^i\ot\xd x^i+\xd y^i\ot\xd y^i\big)}$$
defines a contact metric structure on $\R^{2n+1}$, and we have the tensor
{\cm\be\label{cphi}\phi=\sum^n_{i=1}\Big(\xd x^i\ot\pa_{y^i}-\xd y^i\ot\big(\pa_{x^i}+y^i\pa_z\big)\Big).\ee}
It is easy to check that (\ref{nJ}) defines a complex structure on $\cM$, so the contact metric structure $(\phi,\xi,\eta,g)$ is Sasakian, where $\zx=\pa_z$ is the Reeb vector field in this case.

\mn Another example is $S^{2n+1}$ with {one half of} its canonical Riemannian metric and the contact form being the restriction of the Liouville 1-form on $\R^{2n+2}$,
\be\label{zvy}\zvy=\half\sum_k\big(q^k\xd p_k-p_k\xd q^k\big),\ee
to $S^{2n+1}$; for more details and examples see \cite{Blair:2010, Boyer:2008}. Such odd-dimensional spheres are canonical \emph{contactifications} of complex projective spaces $\C\Pe^n$ with their canonical symplectic forms (cf. \cite{Grabowska:2024}). The latter symplectic manifolds play a fundamental role in quantum physics as manifolds of pure quantum states (for the geometry of quantum mechanics see \cite{Grabowski:2005}).
\end{example}
As we mentioned earlier, Sasakian manifolds in the traditional setting can also be characterized by means of other properties; the following theorem provides an example.
\begin{theorem}\cite{Blair:2010}\label{thm5.4}
    A contact metric structure $(\phi, \xi, \eta, g)$ is Sasakian if and only if
    \begin{align*}
        (\sD_X\phi)(Y)=g(X,Y)\xi-\eta(Y)X,
    \end{align*}
where $\sD$ is the Levi-Civita connection associated with $g$.
\end{theorem}

\mn Another approach, perhaps one of the most instructive ones, is due to Boyer and Galicki, see \cite{Boyer:2008} and references therein, and relates Sasakian geometry with the K\"ahler one. Contact Riemannian manifolds have also been studied in \cite{Tanno:1989}.
\begin{definition} Let $\zw$ be a symplectic form and $g$ be a Riemannian metric on a manifold $\cM$. We say that they are \emph{compatible} if the $(1,1)$ tensor field,
$$J=\big(\zw^\flat\big)^{-1}\circ g^\flat,$$
is an almost complex structure on $\cM$, i.e., $J^2=-\id_{\sT\cM}$.
Here, we consider $\zw^\flat,g^\flat:\sT \cM\to\sT^*\cM$ as given by the contractions in the second argument.
In other words, $J$ is uniquely determined from the identity
\be\label{J}g(X,Y)=\zw\big(X,J(Y)\big).
\ee
\end{definition}
\no The identity (\ref{J}), where $J^2=-\id_{\sT\cM}$, $g$ is a metric, and $\zw$ is a symplectic form, can serve as well as a definition of compatibility for any pair from $\{\zw,g,J\}$ in an obvious way.
In this sense, an \emph{almost K\"ahler manifold} is a manifold $\cM$ equipped with a metric $g$ and a symplectic form $\zw$ (equivalently, with $g$ and $J$, or $\zw$ and $J$) which are compatible. It is a \emph{K\"ahler manifold} if $J$ is integrable, i.e., $J$ is a complex structure on $\cM$. We have the following result \cite{Boyer:2008}, which can be viewed as an alternative definition of a Sasakian manifold in the traditional setting.
\begin{theorem}
Let $(M,g_M,\eta)$ be a contact metric manifold, and let $\cM=M\times \R_+$ be the Riemannian cone over $M$, together with the metric and the symplectic form
{$$g(x,s)=(\xd s)^2+s^2g_M(x), \quad\omega=\half\xd\big(s^2\eta\big).$$}
Then, $(M, g_M, \eta)$ is a Levi (resp., Sasakian) structure if and only if $(\cM,g,\omega)$ is almost K{\"a}hler (resp., K\"ahler).
\end{theorem}
\no In this case, one can look at the relation of K\"ahler structure on $\cM$ and Sasakian structure on $M$, and think of a Levi structure as `almost Sasakian' with the integrability condition being normality of the contact metric structure, which agrees with the terminology used in \cite{Tashiro:1963}.

\section{Riemannian $\Rt$-bundles}
Inspired by the homogeneous approach to contact geometry \emph{via} homogeneous symplectic structures on $\Rt$-bundles, one may try to extend it to the case of Riemannian metrics. The metric $g(x,s)=ds^2+s^2g_M(x)$ on the Riemannian cone $\cM=M\ti\R_+$, that plays a dominant role in the traditional setting of Sasakian manifolds, is $2$-homogeneous with respect to the $\R_+$-action on the cone. On the other hand, the homogeneous symplectic structure associated with any contact structure must be 1-homogeneous. However, a serious problem occurs: a 1-homogeneous symmetric covariant tensor field $g$ on an $\Rt$-bundle cannot be {positive definite}, as $g(x,s)=-g(x,-s)$. We remove this obstacle by starting from a properly defined notion of a 1-homogeneous Riemannian metric.
\begin{definition}
Let $\zt:P\to M$ be an $\Rt$-bundle with the principal $\Rt$-action $s\mapsto h_s$.
A tensor field $K$ on $P$ we call \emph{positively homogeneous} of degree $k\in\Z$ if
\[h_s^*(K)=|s|^k\cdot K\quad\text{for all}\quad s\in\Rt.
\]
In other words,
\[\Ll_\n(K)=k\cdot K\quad\text{and}\quad h_{-1}^*(K)=K.
\]
Here, $h_s^*(K)$ is the pullback of the tensor field $K$ associated with the diffeomorphism $h_s$, and $\n$ is the Liouville vector field.
Covariant tensors that are positively 1-homogeneous we will call simply \emph{positively homogeneous}.
\end{definition}
\no Note that some concepts of homogeneity on $\Rt$-bundles and, more generally, for $G$-structures are already present in a recent paper \cite{Tortorella:2020}.
\begin{example}
On the trivial $\Rt$-bundle $M\ti\Rt$, the forms $|s|^k\za(x)$ are positively homogeneous of degree $k$, the forms $\xd|s|\we\za(x)=\sgn(s)\xd s\we\za(x)$ are positively homogeneous (of degree 1), and functions $A(x)|s|$ are positively homogeneous. In fact, all positively homogeneous functions are of this form.
\end{example}
\no In the sequel, we will use the following description of positively homogeneous symmetric forms. Here, symbols like $\wh\za$ denote the pullbacks of differential forms $\za$ on $M$ to $P$,
\begin{proposition}\label{phforms}
Let $\s$ be a calibration on an $\Rt$-bundle $\zt:P\to M$. Then,
a symmetric $r$-form $\zb$ on $P$ is $\Rt$-invariant if and only if
\be\label{invf}\zb=\sum_{l=0}^r\Big(\frac{\xd\s}{\s}\Big)^l\odot\wh{\za}_{l},\ee
where $\za_l$ is a symmetric $(r-l)$-form on $M$. Here, `$\odot$' denotes the symmetric tensor product, {with the convention $\za\odot\zb=(\za\ot\zb+\zb\ot\za)$, and powers are with respect to this product.}

Moreover, a symmetric form $\zg$ on $P$ is positively homogeneous if and only if $\zg/\s$ is $\Rt$-invariant.
\end{proposition}
\begin{proof}
{Let $k$ be the biggest integer such that $\big(i_\n\big)^k\zb\ne 0$. If $k=0$, then $\zb$ is basic and $\Rt$-invariant, thus a pullback from $M$. Inductively with respect to $k$, we write
$$\zb=\Big(\zb-\big(\xd\s/\s\big)\odot i_\n\zb\Big)+\big(\xd\s/\s\big)\odot i_\n\zb,$$
and as $i_\n\big(\xd\s/\s\big)=1$, we have
$$\big(i_\n\big)^k\big(i_\n\zb\big)=0\quad\text{and}\quad
\big(i_\n\big)^k\Big(\zb-\big(\xd\s/\s\big)\odot i_\n\zb\Big)=0.$$}
The rest is obvious.
\end{proof}
\no One can easily derive an analogous result for skew-symmetric forms, which is actually much simpler, as $(\xd\mathfrak{s}/\mathfrak{s})^{\wedge l}= 0$ for $l>1$.
\begin{corollary} If $\s$ is a calibration on $P$, then any positively homogeneous symmetric 2-form $g$ on $P$ has a unique decomposition
\be\label{metr1}
        g=\wh{A}\dfrac{(\xd \s)^2}{\s}+\xd\s\odot\hm+ \s\cdot \wh{\zg}\,,
\ee
where $A$ is a function, $\zm$ is a 1-form, and $\zg$ is a symmetric 2-form on $M$.
\end{corollary}
\begin{definition}
A \emph{Riemannian $\Rt$-bundle} is an $\Rt$-bundle $\zt:P\to M$ endowed with a positively homogeneous Riemannian metric $g$. Positively homogeneous Riemannian metrics we will call just \emph{homogeneous}.
\end{definition}
\begin{example}
Let us notice that, for every Riemannian metric $g_M$ on $M$, the metric
$$g(x,s)=\frac{\big(\xd s\big)^2}{|s|}+|s|\cdot g_M(x)$$
on the trivial $\Rt$-bundle $P=M\ti\Rt$, which plays the fundamental r\^ole in the Sasakian geometry, is positively homogeneous.
\end{example}
\begin{proposition}\label{gatlas} For every positively homogeneous Riemannian metric $g$ on an $\Rt$-bundle $P$, the function $\s=g(\n,\n)$, where $\n$ is the Liouville vector field on $P$, is a calibration, called the \emph{$g$-calibration}.
\end{proposition}
\begin{proof} Indeed, the Liouville vector field $\n$ is $\Rt$-invariant, thus homogeneous of degree 0, and $g$ is positively homogeneous, so $g(\n,\n)$ is positively homogeneous and positive, as $g$ is Riemannian.

\end{proof}
\begin{proposition}\label{hRm}
A Riemannian metric $g$ on an $\Rt$-bundle $\zt:P\to M$ is positively {homogeneous} if and only if {$\s=g(\n,\n)$ is a calibration, and $g_0=g/\s$ is an $\Rt$-invariant metric on $P$ which reads}
\be\label{metr2}
        g_0=\big(\xd \s/\s+\hm\big)^2+ \hg\,,
\ee
where $g_M$ is a Riemannian metric on $M$, and $\zm$ is a 1-form on $M$.
\end{proposition}
\begin{proof}
Let $\s=g(\n,\n)$. Being positively homogeneous, $g$ is of the form (\ref{metr1}), with $A=1$.
But $g$ is {positive definite}, so for any {horizontal (orthogonal to $\n$)} vector $X\in\sT P$ and any $t\in\R$ we have
$$g(\n+tX,\n+tX)=\s\big(1+2t\cdot\hm(X)+t^2\cdot\wh{\zg}(X,X)\big)>0,$$
thus
$$\hm(X)^2-\wh{\zg}(X,X)<0$$
for $X\ne 0$. Hence, $\wh{\zg}-\hm^2$ is {positive definite} and $\Rt$-invariant, thus of the form $\hg$ for a Riemannian metric $g_M=\zg-\zm^2$ on $M$, and we have (\ref{metr2}).

\end{proof}
\no Note that the Riemannian metric $g_M$ on $M$ is uniquely determined by the positively homogeneous metric $g$ on $P$. We will call it the \emph{shadow} of $g$. {Note also some similarities to the structures appearing in  \cite[Section 6.3]{Tortorella:2020}.}

\mn If we {use another calibration, $\s'=\wh{u}\cdot\s$}, where $u$ is a positive function on $M$, then \be\label{cof}\zz=\xd\s/\s+\hm\ee
will change to
$$\zz'=\xd\s'/\s'+\big({\zm-\xd{u}/{u}}\big)^{\wh{}}.$$
This means that $\zz$ is a connection 1-form of a principal connection on $P$ uniquely determined by the $\Rt$-invariant metric $g_0$ on $P$. In Ehresmann's terms, the connections are given by the {$\Rt$}-invariant horizontal distribution $\cH=\ker(\zz)$ which does not depend on the choice of the calibration.
It is easy to see that $\cH$ is simply the $g_0$-orthogonal complement of the vertical subbundle $\sV P$ (spanned by the Liouville vector field $\n$) and that $g_0(\n,\n)=1$.
\begin{definition}\label{gcal}
A positively homogeneous Riemannian metric $g$ on an $\Rt$-bundle $\zt:P\to M$ we call \emph{calibrated} if $\zz=\zz_\s=\xd\s/\s$ is the $\s$-connection (cf. Remark \ref{conn}), where $\s=g(\n,\n)$ is the $g$-calibration. In other words, $\zm=0$ and
\be\label{cc} g=\s\Big(\big(\xd\s/\s\big)^2+\hg\Big),\ee
for some Riemannian metric $g_M$ on $M$, the shadow of $g$.
\end{definition}
\section{K\"ahlerian $\Rt$-bundles}\label{Sec6}
To consider K\"ahlerian $\Rt$-bundles, let us observe first that given an almost contact structure $(\phi, \xi, \eta)$ on $M$, we can define a $(1,1)$-tensor $J$ on $\cM=M\ti\R_+$ by
\begin{align}\label{J2}
    J(X)=\phi(X)+\eta(X)\n,\,\,\, J(\n)=-\xi,
\end{align}
where $\n=s\partial_s$ is the generator of the $\R_+$-action on the cone $\cM$. It is easy to see that $J$ is an almost complex structure on $\cM$ and the tensor $J$ is $\R_+$-invariant; in particular,  $\Ll_\n J=0$. Since on $\Rt$-bundles we consider Riemannian metrics which are positively homogeneous, and symplectic forms which are homogeneous, a proper generalization of $J$ to $\Rt$-bundles should be a $(1,1)$-tensor being $\Rt$-invariant up to a sign.
\begin{definition}
Let $\zt:P\to M$ be an $\Rt$-bundle with the principal $\Rt$-action $s\mapsto h_s$.
A tensor field $K$ on $P$ we call \emph{half-invariant} if
\[ h_s^*(K)=\sgn(s)\cdot K\quad\text{for all}\quad s\in\Rt.
\]
For a $(1,1)$-tensor $J:\sT P\to\sT P$, this is equivalent to
$$\big(\sgn(s)\cdot J\big)\circ\sT h_s=\sT h_s\circ J.$$
\end{definition}
\no It is easy to see that, on trivial $\Rt$-bundles, $K$ is half-invariant if and only if $K=\sgn(s)K_0$, where { $K_0$} is $\Rt$-invariant.
\begin{example}
On the trivial $\Rt$-bundle $M\ti\Rt$ the form $\xd s/|s|$ is half-invariant.
\end{example}
\begin{definition}
An \emph{almost K\"ahlerian $\Rt$-bundle} is a symplectic $\Rt$-bundle $(P,\zw)$, equipped additionally with a compatible almost complex half-invariant structure $J$. Such a structure we call \emph{K\"ahlerian $\Rt$-bundle} if $J$ is integrable. (Almost) \emph{K\"ahlerian $\R_+$-bundles} are defined analogously; we replace the half-invariance condition with the invariance.
\end{definition}
\no Let us recall that the compatibility means that the tensor field $g$ on $P$, defined by
\be\label{hK}g(X,Y)=\zw\big(X,J(Y)\big),\ee
is a Riemannian metric. The calibration $\s=g(\n,\n)$ can be described in terms of $\zw$ and $J$ as $\s=\zw\big(\n,J(\n)\big)$. In the case of $\zw$ being 1-homogeneous, this Riemannian metric is automatically positively homogeneous.  Since $J^2=-\id_{\sT P}$, {due to compatibility,} it is automatically an isometry and symplectomorphism,
\be\label{sym} {g}\big(J(X),J(Y)\big)=g(X,Y),\quad\zw\big(J(X),J(Y)\big)=\zw\big(X,Y\big).
\ee
\begin{proposition}\label{prop 7.5}
An (almost) K\"ahlerian $\Rt$-bundle can be equivalently defined as a symplectic $\Rt$-bundle $(P,\zw)$ endowed with a positively homogeneous Riemannian metric $g$ on $P$, which is \emph{compatible} with $\zw$, i.e., the $(1,1)$-tensor $J$ determined by (\ref{hK})
is an (almost) complex structure on $P$. In this case, $J$ is automatically half-invariant.
\end{proposition}

\begin{remark}
    Let $(M, J, g)$ be a Hermitian manifold, i.e., $J$ is a complex structure and $g$ is a Hermitian metric. The manifold $(M,J,g)$ is called a \emph{locally conformal K\"ahler manifold} if there exist an open cover $\{U_i\}_{i\in I}$ of $M$ and a family $\{f_i\}_{i\in I}$ of smooth functions $f_i\colon U_i\raa \R$, such that the local metrics
    \be\label{cmet}
    g_i=e^{-f_i}g_{|_{U_i}}
    \ee
are K\"ahler with respect to $J$.

\mn Now, let $\Omega={-g}(X, JY)$ be the associated $2$-form (resp., let $\Omega_i$ be associated with $g_i$). From (\ref{cmet}) we obtain $\Omega_i=e^{-f_i}\Omega_{|_{U_i}}$. The $2$-form $\Omega$ is therefore nondegenerate but generally not closed, $\xd\zW=\za\we\zW$. This leads to {the concept} of a \emph{locally conformal symplectic form} (cf. \cite{Vaisman:1976}), {which} has been extensively studied since the 1970s (see \cite{Ornea:2024} and references therein).
In \cite{Vaisman:1980}, Vaisman investigated conformal changes of an almost contact structure $(\phi, \xi, \eta, g)$, i.e., changes of the form
    \be\label{conf}
    \phi^\prime=\phi,\,\,\,\xi^\prime= e^f\xi,\,\,\,\eta^\prime=e^{-f}\eta,\,\,\,g^\prime=e^{-2f}g,
    \ee
and discussed a notion of \emph{locally conformal Sasakian manifold}.

\mn In the above locally conformal approach to K\"ahler structures, the metric is fixed, and we deform conformally the associated 2-form to get local K\"ahler structures. In our case, the situation is different. We have a fixed symplectic form on a symplectic cover of a contact manifold, and we look locally for a K\"ahler metric satisfying the homogeneity assumption additionally.
\end{remark}
\begin{example}\label{ex7.1}
Let us go back to the M\"obius band $B\to S^1$ and the corresponding symplectic $\Rt$-bundle $\sT^*B^\ti$ from Example \ref{jetm}.
We have two charts, $\tO$ and $\tU$, on $\sT^*B^\ti$, equipped with coordinates $(x,s,p,z)$ and $({x'},s',p',z')$ taking values in
$]0,1[\ti\Rt\ti\R\ti\R$ and $]1/2,3/2[\ti\Rt\ti\R\ti\R$, respectively. {These are not the adapted coordinates, but the coordinates coming from trivializations of the $\Rt$-bundle $\sT^*B^\ti\to\sJ^1B^*$ and coordinates $(x,s,p)$ on $\bar\cO$ and $({x'},s',p')$ on $\bar\cU$.}

\mn The transition map is
$$(x',s',p',z')=
\begin{cases}
(x,s,p,z)\quad\text{if}\quad x\in]1/2,1[\\
(x+1,-s,-p,-z)\quad\text{if}\quad x\in]0,1/2[\,,
\end{cases}
$$
and the canonical {homogeneous} symplectic form $\zw$ in these coordinates reads
\be\label{zw}\zw=\xd s\we\zh+s\cdot\xd\zh=\xd s'\we\zh'+{s'}\cdot\xd\zh',\ee
where $\zh=\xd z-p\xd x$ on $\tO$, and $\zh'=\xd z'-p'\xd x'$ on $\tU$.

\mn It is easy to see that the Riemannian metric $g$, given in $\tO$ by
$$g=\frac{\xd s^2}{|s|}+|s|\big((\xd p)^2+(\xd x)^2+\zh^2\big),$$
coincides on the intersection $\tO\cap\tU$ with the Riemannian metric $g'$, given in $\tU$ by
$$g'=\frac{(\xd s')^2}{|s'|}+|s'|\big((\xd p')^2+(\xd x')^2+(\zh')^2\big).$$
Consequently, we have a well-defined Riemannian metric on $\sT^*B^\ti$, which will be denoted with $g$. This Riemannian metric is positively homogeneous.
\begin{proposition}\label{prop7.7}
The Riemannian metric $g$ is compatible with the symplectic form $\zw$, and the $(1,1)$-tensor field
$$J:\sT\sT^*B^\ti\to\sT\sT^*B^\ti,\quad J=\big(\zw^\flat\big)^{-1}\circ g^\flat,$$
is a complex structure. In other words, the ingredients of $(\zw,g,J)$ on $\sT^*B^\ti$ give rise to a K\"ahlerian $\Rt$-bundle.
\end{proposition}
\begin{proof}
We will work with local coordinates in $\tO$. The form of $\zw$ in (\ref{zw}) suggests using $\xd s/s,\zh,\xd p,\xd x$ as a basis of invariant 1-forms on $\tO$. The dual basis of vector fields is $\n=s\pa_s,\pa_z,\pa_p,X=\pa_x+p\pa_z$. Let us recall that $\pa_p$ and $X$ span $\ker(\zh)$. We have
\beas
&& g^\flat(\n)=\sgn(s)\xd s,\ g^\flat(\pa_z)=|s|\zh,\ g^\flat(\pa_p)=|s|\xd p,\ g^\flat(X)=|s|\xd x;\\
&& \zw^\flat(\n)=-s\,\zh,\ \zw^\flat(\pa_z)=\xd s,\ \zw^\flat(\pa_p)=s\,\xd x,\ \zw^\flat(X)=-s\,\xd p.
\eeas
Hence, $J=\big(\zw^\flat\big)^{-1}\circ g^\flat$ acts as follows:
$$J(\n)=\sgn(s)\pa_z,\ J(\pa_z)=-\sgn(s)\n,\ J(\pa_p)=-\sgn(s)X,\ J(X)=\sgn(s)\pa_p\,.
$$
In other words,
\be\label{J1} J=\sgn(s)\Big(\frac{\xd s}{s}\ot\pa_z+\xd x\ot\pa_p-\xd p\ot X-s\zh\ot\pa_s\Big)\,.
\ee
The transformation $(x,s,p,z)\mapsto(x+1,-s,-p,-z)$ leaves the above $(1,1)$-tensor field $J$ invariant, so (\ref{J1}) defines properly a globally defined tensor field $J$ on $\sT^*B^\ti$. It is easy to see that $J^2=-\id$, so we deal with an almost complex structure.

The complexified tangent bundle $\sT^\C\sT^*B^\ti$ splits into two $J$-invariant subbundles $V_\pm$, $\sT^\C\sT^*B^\ti=V_+\op V_-$. The subbundle
$V_+$ consists of eigenvectors with the eigenvalue $i$, and is spanned by the complex vector fields $A_1=\sgn(s)\pa_z+i\n$ and $A_2=\sgn(s)\pa_p+iX$, while the subbundle $V_-$ consists of eigenvectors with the eigenvalue $-i$, and is spanned by the complex vector fields $B_1=\n+i\sgn(s)\pa_z$ and $B_2=X+i\sgn(s)\pa_p$.
It is easy to check that all the vector fields {$A_1,A_2$ commute, so do $B_1,B_2$, thus} $V_\pm$ is involutive, and $J$ is a complex structure.

\end{proof}
\end{example}

\section{Generalized Sasakian manifolds}
Instead of proposing an \emph{ad hoc} definition of a Sasakian-like structure on a general contact manifold $(M,C)$, i.e., a manifold endowed with a contact distribution, we will derive a conceptual approach to this question \emph{via} homogeneous K\"ahler structures on $\Rt$-bundles. To this end, we will characterize almost K\"ahlerian and K\"ahlerian $\Rt$-bundles in terms of structures on the base contact manifold. To simplify the notation, we will denote covariant tensors on $M$ and their pullbacks to fiber bundles with the same symbol. Let us start with the cooriented case.

\subsection{K\"ahlerian $\R_+$-bundles}\label{sec8}
To simplify notation, let us fix the convention that for the trivial $\R_+$-principal bundle $\cM=M\ti\R_+$ with the vertical coordinate $s\in\R_+$, we will naturally identify vector fields $X$ on $M$ with their pullbacks $X(x,s)=X(x)$ on $\cM$, and differential forms on $M$ with their pullbacks to $\cM$. These pullbacks are $\R_+$-invariant.

\mn Let $\zt:\cM\to M$ be a principal $\R_+$-bundle equipped with a homogeneous symplectic structure $\zw$ and a homogeneous Riemannian structure $g$.

According to (\ref{sffc}) and Proposition \ref{hRm}, in the trivialization $\cM=M\ti\R_+$ associated with the calibration function $s=g(\n,\n)$, the symplectic and Riemannian forms read $\zw=s\zw_0$ and $g=sg_0$, where
\be\label{sss}\zw_0(x,s)=\big(\xd s/s\big)\we\zh(x)+\xd\zh(x),
\ee
and
\be\label{Sas} g_0(x,s)=\big(\xd s/s+\zm(x)\big)^2+ g_M(x)\,,
\ee
respectively, where $\zh$ is a contact form on $M$, $\zm$ is a 1-form on $M$, and $g_M$ is a Riemannian metric on $M$.
Put $C=\ker(\zh)$ to be the contact distribution associated with $\zh$. Let $\cH=\n^\perp$ be a horizontal distribution on $\cM$, being the orthogonal complement of the vertical distribution. In other words, $\cH$ represents a principal connection on the $\R_+$-bundle $\cM$ with the connection 1-form $\zz(x,s)=\xd s/s+\zm(x)$ (cf. (\ref{cof})).
The horizontal lift $X^h$ of a vector field $X$ on $M$ is given by
\be\label{hor}X^h=X-\zm(X)\n.\ee
Denote with $C^h$ the distributions spanned by the horizontal lifts of vector fields from $C$, and with $g_C$ the restriction of the metric $g_M$ to $C$.

\mn Writing $g_0$ and $\zw_0$ in terms of generators $X^h$ and $\n$ of $\sT\cM$, where $X$ are vector fields on $M$, we easily get the following.
\begin{proposition}\label{fromulae}
The metric {$g_0$ and the 2-form $\zw_0$} are uniquely characterized by the formulae:
\begin{align*}
& g_0(\n,\n)=1;\\
& g_0(\n,X^h)=0;\\
& g_0(X^h,Y^h)=g_M(X,Y);\\
& \zw_0(X^h,Y^h)=\big(\xd\zh{-}\zm\we\zh\big)(X,Y);\\
& \zw_0(\n,X^h)=\zh(X).
\end{align*}
Moreover, $[\n,X^h]=0$ and
\be\label{brc} [X^h,Y^h]=[X,Y]^h{-}\xd\zm(X,Y)\n.
\ee
\end{proposition}
\no Note that $\xd\zm$ is the curvature form of the connection.

\mn Define now a $(1,1)$-tensor $J:\sT\cM\to\sT\cM$ from the identity,
\be\label{e1}g(X,Y)=\zw\big(X,J(Y)\big).\ee
The compatibility of $g$ and $\zw$ reads $J^2=-\id$, and means that our structure $(\cM,g,\zw)$ is homogeneous almost K\"ahlerian. Clearly, it is a K\"ahler structure if and only if, additionally, $N_J=0$.

\mn The following theorem describes almost K\"ahlerian $\R_+$-bundles.
\begin{theorem}\label{main1} We have $J^2=-\id$ if and only if $J^2(\n)=-\n$ and there is a complex structure $\pc:C\to C$, $\pc^2=-\id_C$, on the vector bundle $C$, such that
\be\label{pc} J(X^h)=\bl\pc(X)\br^h,
\ee
for all $X\in C$.

\mn In this case, $\pc$ is a Levi structure on the contact manifold $(M,\zh)$, with the Levi form equal to $g_C$,
\be\label{lev}
g_C(X,Y)=\xd\zh\big(X,\pc(Y)\big),
\ee
and $J(\n)=A^h$, where the vector field $A$ on $M$ is uniquely determined by the conditions
\be\label{u}
\xd\zh(A,X)={-}\zm(X),\quad \zh(A)=1,
\ee
for all $X\in C$. Equivalently,
\be\label{u1}
g_M(A,\cdot)=\zh,\quad \zh(A)=1.
\ee
In particular, the splitting $\sT\cM=W\op C^h$, with $W=\on{span}\bk{\n}{A^h}$, is $J$-invariant and $g$-orthogonal, and
\be\label{aK}\xd\zh\big(\pc(X),\pc(Y)\big)=\xd\zh(X,Y).\ee
\end{theorem}
\begin{proof} Suppose $J^2=-\id$. In particular, $J^2(\n)=-\n$ and out of (\ref{e1}),
$$g\bl J(\n),\n\br=\zw\bl J(\n),J(\n)\br=0,$$
so $J(\n)$ is $g$-orthogonal to $\n$. For $X\in C$, we also have
$$0=-\zh(X)=\zw\bl X^h,\n\br=g\bl X^h,-J(\n)\br,$$
so $J(\n)$ is $g$-orthogonal to $C^h$, thus
$$W=\on{span}\bk{\n}{J(\n)}$$
is orthogonal to $C^h$. The latter is $J$-invariant and $J$ is an isometry, thus $C^h$ is also $J$-invariant,
$$ J(C^h)=C^h.$$
Therefore, we have a $g$-orthogonal splitting
\be\label{spli}
\sT\cM=W\oplus C^h=\la \n\ran\oplus\la J(\n)\ran\oplus C^h.
\ee
If $X\in C$, then $J(X^h)\in C^h$, and $J(X^h)$ is $\R_+$-invariant ($J$ respects the invariance), so $J(X^h)=Y^h$ for some $Y\in C$. This implies that there is a VB-morphism $\pc:C\to C$ such that $J(X^h)=\bl\pc(X)\br^h$ and $\pc^2=-\id_C$. The converse is obvious, as $J^2(\n)=-\n$ and $J(X^h)=\bl\pc(X)\br$ imply $J^2=-\id$.

\mn Let us assume now that $J^2=-\id$. We know that $J(\n)\in\cH$, so due to the invariance, $J(\n)=A^h$ for some vector field $A$ on $M$. Moreover, $A$ is $g_M$-orthogonal to $C$. Consequently, for $X,Y\in C$,
$$g_C(X,Y)=g_M(X,Y)=g_0(X^h,Y^h)=\zw_0\bl X^h,\pc(Y)^h\br=\xd\zh\bl X,\pc(Y)\br,$$
that proves (\ref{lev}), and (\ref{aK}) follows automatically (see Proposition \ref{pin}(a)). We also have,
\begin{align*} &\zw(X^h,A^h)=\zw\bl X^h,J(\n)\br=g(X^h,\n)=0,\\
&\zh(A)=\zw_0(\n,A^h)=\zw_0\bl\n,J(\n)\br=g_0(\n,\n)=1,
\end{align*}
that defines {$A^h\in\cH$} as the unique vector field $\zw$-orthogonal to {$C^h$} satisfying
\be\label{B}\zw(\n,A^h)=\zh(A)=1.
\ee
In particular, for $X\in C$, we have
$$0=\zw(X^h,A^h)=\bl\xd\zh{-}\zm\we\zh\br(X,A)=\xd\zh(X,A){-}\zm(X),$$
so
\be\label{B1} \xd\zh(X,A)=\zh([A,X])={\zm(X)},
\ee
for all $X\in C$, and
$$\|A\|^2=g_M(A,A)=g_0(A^h,A^h)=\zw_0\bl A^h,J(A^h)\br=\zw_0(A^h,-\n)=\zh(A)=1,$$
thus $A$ is the unit vector in the metric $g_M$ satisfying $g_M(A,\cdot)=\zh$.

\end{proof}
\no Note that the restriction $\zm_C$ of $\zm$ to $C$ is determined by the metric,
{$$g_M(A,[A,X])=\zm(X).$$}
The above theorem gives a one-to-one correspondence between homogeneous almost-K\"ahler structures on the trivial $\R_+$-bundles $\cM=M\ti\R_+$ and geometric structures on $M$ consisting of a contact form $\zh$, a Riemannian metric $g_M$, a 1-form $\zm$, and a Levi structure $\pc$ on $(M,\zh)$, satisfying the following conditions:
\begin{align}\label{M1}
& g_C(X,Y)=\xd\zh\bl X,\pc(Y)\br;\\
& g_M=\zh\ot\zh+\gc,\label{M2}
\end{align}
where $\gc$ is the unique extension of $g_C$ to a symmetric tensor on $\sT M$ vanishing on $L=C^\perp$. This suggests the following concept generalizing contact metric structures (cf. Definition \ref{g-as}).
\begin{definition}
\
\begin{itemize}
\item Let $g_M$ be a Riemannian metric on a cooriented contact manifold $(M,\zh)$, and let $L=C^\perp$ be the line subbundle in $\sT M$ which is $g_M$ orthogonal to $C=\ker(\zh)$. We call $g_M$ a \emph{Levi metric} if $$\|Z\|^2=g_M(Z,Z)=|\zh(Z)|^2$$
    for any vector field $Z$ taking values in $L$, and the map $\pc:C\to C$, defined for vector fields $X,Y\in C$ by $$g_M(X,Y)=\xd\zh\bl X,\pc(Y)\br,$$
    is a complex structure on $C$, $\pc^2=-\id_C$. We will call $\pc$ the \emph{associated Levi structure}.

    \mn A contact manifold $(M,\zh)$ equipped with a Levi metric $g_M$ we call a \emph{contact metric structure}, and the triple $(M,\zh,g_M)$ a \emph{contact metric manifold}. We call the vector field $A$ the \emph{split vector field}.
\item A \emph{split Levi structure} on a contact manifold $(M,\zh)$ is a $(1,1)$-tensor $\bp:\sT M\to\sT M$ such that $\bp(\sT M)=C=\ker{\zh}$, and whose restriction $\pc$ to $C$ is a Levi structure. In this case, we call the triple $(M,\zh,\bp)$ a \emph{split Levi manifold}.
    \end{itemize}
\end{definition}
\no Any Levi metric $g_M$ on a contact manifold $(M,\zh)$ is of the form (\ref{M2}), where $\gc$ is the unique extension of the tensor field $g_C$--the restriction of $g_M$ to $C$, vanishing on $L=C^\perp$. Moreover, $g_C$ satisfies (\ref{M1}), i.e., $g_C$ is the Levi form of the associated Levi structure $\pc$.

\begin{proposition} On any contact metric manifold $(M,\zh,g_M)$, there exists a distinguished vector field $A$ taking values in $L=C^\perp$ such that $\zh(A)=1$, and a distinguished 1-form $\zn$ determined by $\zn={-}i_A\xd\zh$. 

\mn Moreover, conditions (\ref{M1}) and (\ref{M2}) define a canonical one-to-one correspondence between split Levi structures $\bp$ and Levi metrics $g_M$ satisfying $\ker(\bp)=C^\perp$. More precisely, $\bp\big|_C=\pc$ and $\bp(A)=0$. Hence, contact metric structures are the same as split Levi structures.
\end{proposition}
\no The split Levi structure $\bp$ we call \emph{associated with $g_M$}, the vector field $A$--the \emph{split vector field}, the 1-form $\zn$--the \emph{associated 1-form}, and any 1-form $\zm$ on $M$ extending $\zn\big|_C$ on $C$--a \emph{compatible connection}. Now, we can reformulate Theorem \ref{main1} as follows.
\begin{theorem}\label{akt}
Let $(M,\zh,g_M)$ be a contact metric manifold and $\zm$ be a compatible connection. Then the symplectic form $\zw$ and the Riemannian metric $g$, defined on the trivial $\R_+$-bundle $P=M\ti\R_+$ by
\bel\label{mmain1} &\zw(x,s)=\xd s\we\zh(x)+s\cdot\xd\zh(x);\\
& g(x,s)=s\Bl\big(\xd s/s+\zm(x)\big)^2+ g_M(x)\Br,\nn
\eel
constitute a homogeneous almost-K\"ahler structure. Moreover, any homogeneous almost K\"ahler structure on an $\R_+$-principal bundle is of this form in a certain trivialization.
\end{theorem}

\mn Now, let us pass to studying the integrability condition for homogeneous almost K\"ahler structures on $\cM=M\ti\R_+$.
\begin{theorem}\label{main2}
Suppose (\ref{mmain1}) {defines} an almost K\"ahler structure. This structure is K\"ahlerian (integrable) if and only if the connection 1-form $\zm$, the vector field $A$, and the extension $\bp:\sT M\to\sT M$ of $\pc$ with $i_A\bp=0$, satisfy the conditions
\bel\label{S1}
& N_\bp(X,Y)+\xd\zh(X,Y)A=0;\\
& N_\bp(A,X)=0;\label{S2}\\
& {\xd\zm(A,X)+\zm\bl\pc(X)\br=0},\label{S3}
\eel
{for all $X,Y\in C$.}
\end{theorem}
\begin{proof}
The condition for the almost K\"ahlerian structure to be K\"ahlerian is the vanishing of the Nijenhuis torsion $N_J$. Since $\sT\cM=W\op C^h$ is a $J$-invariant splitting, we can discuss $J_W=J\big|_W$ and $J_C=\phi_C$ separately, and then look at $N_J(X,Y)$ for $X\in C$ and $Y\in W$. We will also use the $g_M$-orthogonal decomposition
$$\sT M=C\oplus\la A\ran,$$
and decompose any vector field $X$ on $M$ as
$$X=X_0+g_M(A,X)A=X_0+\zh(X)A.$$
In particular (cf. (\ref{B1})), for any $X,Y\in C$,
\begin{align}\label{split}
&[A,X]=[A,X]_0+\zh\bl[A,X]\br A=[A,X]_0{+}\zm(X) A;\\
&[X,Y]=[X,Y]_0+\zh\bl[X,Y]\br A=[X,Y]_0-\xd\zh(X,Y)A.\label{split1}
\end{align}
\no Since $W$ has rank 2, $\cN_{J_W}$ is automatically $0$, as $N_J\big(X,J(X)\big)$ is always $0$.

\mn {Next}, let us compute the Nijenhuis torsion $N_J(X^h,Y^h)$ for $X,Y\in C$. We have,
\begin{align}\label{t1}
&[J(X^h),J(Y^h)]-[X^h,Y^h]=[\pc(X)^h,\pc(Y)^h]-[X^h,Y^h]\\
&=\Bl[\pc(X),\pc(Y)]-[X,Y]\Br^h{-}\Bl\xd\zm\bl\pc(X),\pc(Y)\br-\xd\zm(X,Y)\Br\n.\nn
\end{align}
Because $\pc$ is a Levi structure on $(M,\zh)$, it is an almost CR structure and we have (\ref{acr}),
$$[\pc(X),\pc(Y)]-[X,Y]\in C.$$
Similarly, we get
\begin{align*}
&[J(X^h),Y^h]+[X^h,J(Y^h)]\\
&=\Bl[\pc(X),Y]+[X,\pc(Y)]\Br^h{-}\Bl\xd\zm\bl\pc(X),Y\br{+}\xd\zm\bl X,\pc(Y)\br\Br\n,
\end{align*}
where
$$[\pc(X),Y]+[X,\pc(Y)]\in C,$$
so
\begin{align}\label{t2}
&J\Bl[J(X^h),Y^h]+[X^h,J(Y^h)]\Br=\Bl\pc\bl[\pc(X),Y]+[X,\pc(Y)]\br\Br^h\\
&{-}\Bl\xd\zm\bl\pc(X),Y\br+\xd\zm\bl X,\pc(Y)\br\Br A^h.\nn
\end{align}
Comparing (\ref{t1}) with (\ref{t2}), we get
\begin{align}\label{(c)}
&[\pc(X),\pc(Y)]-[X,Y]=\pc\bl[\pc(X),Y]+[X,\pc(Y)]\br;\\
&\xd\zm\bl\pc(X),\pc(Y)\br=\xd\zm\bl X,Y\br.\label{(d)}
\end{align}
The third identity,
$$
\xd\zm\bl\pc(X),Y\br+\xd\zm\bl X,\pc(Y)\br=0,
$$
is equivalent to (\ref{(d)}).

\mn Note that (\ref{(c)}) simply means (cf. (\ref{Nt})) that the almost CR structure $(C,\pc)$ is integrable, $\cN_{\pc}=0$. Denoting with $\bp:\sT M\to\sT M$ the extension of $\pc$ satisfying the condition $\bp(A)=0$, we also have
\be\label{nfc} \cN_\bp(X,Y)=0.
\ee
Following the standard Sasaki condition (\ref{Ncc}), we can rewrite (\ref{(c)}) as
\be\label{Nbc}
N_\bp(X,Y)=-\xd\zh(X,Y)A.
\ee
It remains to consider `mixed terms', so to check under what conditions $N_J(X^h,\n)=0$ for all $X\in C$. We have
\beas N_J(X^h,\n)&=&\big([J(X^h),J(\n)]-[X^h,\n]\big)-J\big([J(X^h),\n]+[X^h,J(\n)]\big)\\
&=&[J(X^h),J(\n)]-J\big([X^h,J(\n)]\big)\\
&=&[\big(\pc(X)\big)^h,A^h]-J\big([X^h,A^h]\big),
\eeas
as $\n$ commutes with all vector fields of the form $X^h$. We have
\begin{align}\label{first}
[\pc(X)^h,A^h]&=[\pc(X),A]^h{-}\xd\zm\bl\pc(X),A\br\n\\
&=[\pc(X),A]_0^h-\zm\bl\pc(X)\br A^h{-}\xd\zm\bl\pc(X),A\br\n.\nn
\end{align}
Similarly,
\begin{align}\label{second}
J[X^h,A^h]&=J\Big([X,A]_0^h-\zm(X)A^h{-}\xd\zm(X,A)\n\Big)\\
&={\bl\pc([X,A]_0)\br^h}+\zm(X)\n{-}\xd\zm(X,A)A^h.\nn
\end{align}
Combining (\ref{first}) and (\ref{second}), we get,
\begin{align}\label{aa}
&[\pc(X),A]_0=\pc([X,A]_0);\\
&{\xd\zm(X,A)}=\zm\bl\pc(X)\br.\label{bb}
\end{align}
The third identity,
$$\zm(X)={-}\xd\zm\bl\pc(X),A\br,$$
is equivalent to (\ref{bb}) under the substitution $X:=\pc(X)$, and (\ref{aa}) can be viewed as an analog of (\ref{Kcond}). Alternatively, (\ref{aa}) is equivalent to
\be\label{nbp}{N_\bp(A,X)=0},
\ee
for all $X\in C$. Indeed, as $\bp(A)=0$, we have
$${N_\bp(A,X)}=-\bp\bl[\bp(X),A]-\bp([X,A])\br,
$$
and
$$\bp\Bl [\bp(X),A]-\bp([X,A])\Br={\pc\Bl[\pc(X),A]_0-\pc([X,A]_0)\Br}=0.
$$
Let us observe that condition (\ref{aa}) implies (\ref{(d)}), making the latter superfluous. Indeed, we also have
$$[A,X]=[A,X]_0{+}\zm(X)A, \quad\text{and}\quad [X,Y]=[X,Y]_0-\xd\zh(X,Y)A.$$
With this notation, the $A$-component of the Jacobi identity
$$[A,[X,Y]]+[X,[Y,A]]+[Y,[A,X]]=0,$$
gives
\begin{equation}\label{eqj}
\xd\zm(X,Y)=-A\bl\xd\zh(X,Y)\br+\xd\zh([A,X]_0,Y)+\xd\zh(X,[A,Y]_0)+\zm(A)\xd\zh(X,Y).
\end{equation}
Now replace $X,Y$ by $\pc(X),\pc(Y)$ in (\ref{eqj}) and use both (\ref{aa}) and (\ref{aK}) to get
$$\xd\zm\bl\pc(X),\pc(Y)\br=\xd\zm(X,Y).$$
The complete set of conditions assuring integrability of the almost-K\"ahler structure is then exactly (\ref{S1}), (\ref{S2}), and (\ref{S3}).

\end{proof}
\no Note that (\ref{S1}) and (\ref{S2}) can be unified in one condition,
\be\label{Slike}\fn_\bp:=N_\bp+\Bl\xd\zh\circ(\bp\ot\bp)\Br\ot A=0.\ee
We will call $\fn_\bp$ the \emph{modified Nijenhuis tensor} of the split Levi structure $\bp$ (or the corresponding contact metric structure).
\begin{remark}
Note that finding the connection 1-form appearing in the K\"ahler case is not trivial. The part $\zm_C$ is known because of (\ref{B1}), $\xd\zh(A,X))=\zm(X)$. However, (\ref{S3}) puts strong restrictions on a possible extension. Of course, any extension is determined by fixing the value $f=\zm(A)$. Condition (\ref{bb}) is then a differential equation for $f$. More precisely, writing
$$
\zm=\zm_C+f\eta ,
$$
(\ref{bb}) becomes
$$X(f)=A(\zm_C(X))-\zm_C([A,X]_0)-f\zm_C(X)+\zm_C\bl\pc(X)\br,
$$
where $X$ runs over all sections of $C$. Define
$$\gamma(X)=A(\zm_C(X))-\zm_C([A,X]_0)+\zm_C\bl\pc(X)\br.
$$
Then $\zg$ is a section of $C^*$ and the equation for $f$ reads
\be\label{ef}
\xd_Cf+f\zm_C=\gamma .
\ee
Since the contact distribution $C$ is bracket-generating, this equation is overdetermined, and the existence and the uniqueness of solutions are open and nontrivial questions.
\end{remark}
\begin{example}\label{eflat} The case $\zm\big|_C=0$ is of particular interest. We have $A=\zx$ and $\zm=a\zh$, for a function $a$. One can easily prove that $a$ must be constant.  Condition (\ref{S3}) is trivially satisfied, although the curvature $\xd\zm=a\xd\zh$ can be nontrivial. {If the connection is flat, $\xd\zm=0$, then (\ref{bb}) implies that $\zm$ vanishes on $C$ and $a=0$, so $\zm=0$. Hence, we are in the situation of the classical Sasaki structure.}
\end{example}
\no Following the idea of Boyer and Galicki \cite{Boyer:2008} to identify Sasaki structures with some K\"ahler structures, we propose the following.
\begin{definition} A \emph{generalized Sasaki structure} is a split Levi structure $(M,\zh,\bp)$ equipped with a compatible connection 1-form $\zm$ such that $\fn_\bp=0$ and $\zm$ satisfies (\ref{S3}), where $A$ is the corresponding split vector field. Such connection 1-forms we call \emph{K\"ahlerian connections}.
\end{definition}
\no In this terminology, Theorem \ref{main2} just says that there is a one-to-one correspondence between generalized Sasaki structures on a contact manifold $(M,\zh)$ and homogeneous K\"ahler structures on its symplectic cover $\cM=M\ti\R_+$. Note that the presence of the connection 1-form $\zm$ is in this definition unavoidable. We have shown that homogeneous K\"ahler structures necessarily involve principal connections which, therefore, are part of the game. If one wants to work without connections as an additional input, then one could remain with the flat case $\zm=0$, thus with the standard Sasaki structures.

\section{The case of general contact structures}
For a general contact manifold $(M,C)$, where $C$ is a contact distribution, there exists, generally, no global contact form $\zh$ such that $C=\ker(\zh)$, and no such contact form is distinguished, even if the contact manifold is trivializable (coorientable). However, in the presence of a Riemannian metric $g_M$ on $M$ the situation changes significantly because of the existence of a canonical calibration (see Section \ref{sec-cal}). Let us start with the following observation.
\begin{proposition}\label{clb} Let $(M,C)$ be a contact manifold and $(P,\zw)$ be its symplectic cover. Any Riemannian metric $g_M$ on $M$ determines a unique calibration $\s$ on $P$, thus a submanifold $\tM$ of $P$ defined as the 0-locus of the function $\s-1$. This submanifold is a 2-sheet cover of $M$ with respect to the bundle projection $\zt:P\to M$, and the multiplication $h_{-1}$ by $(-1)$ in $P$ induces a canonical diffeomorphism $\zs:\tM\to\tM$ intertwining the two local sheets.

\mn The restriction $\tzh$ of the Liouville 1-form $\zvy=i_\n\zw$ to $\tM$ is a genuine contact form on $\tM$ satisfying $\zs^*(\tzh)=-\tzh$, whose contact distribution $\tC$ is the pullback of $C$, $\sT\zt(\tC)=C$. Hence, the contact form $\tzh$ can be viewed as a paired contact form $|\zh|$ on $M$.
\end{proposition}
\no The cooriented contact manifold $(\tM,\tzh)$ (equivalently, the paired contact form $|\zh|$ on $M$) we will call the \emph{$g_M$-induced trivialization of $(M,C)$}, and the calibration $\s$--the \emph{$g_M$-calibration} (cf. Theorem \ref{calibration}).
\begin{proof}
We use the metric $g_M$ to define the orthogonal complement $L=C^\perp$ of the contact distribution, which is a rank-1 distribution canonically isomorphic with the line bundle $\sT M/C$. The metric $g_M$ restricted to $L$ yields, in turn, a VB-metric on $L$, thus an isomorphism $L\simeq L^*$. Since $P=\bl L^*\br^\ti$ (cf. Theorem \ref{main}), the corresponding metric norm $\|\cdot\|_{L^*}$ on $L^*$ defines a calibration $\s$ on $\zt:P\to M$: for $p\in P_x\subset L^*_x$, we have
$$\s(p)=\frac{|p(v_x)|}{\nm{v_x}},$$
where $v_x$ is any vector in $L^\ti_x$.

\mn The calibration defines a pair of local trivializations,
$$\phi_\pm:{\zt^{-1}(U)\to U\ti\Rt},\quad \phi_\pm(p)=\bl\zt(p),\pm s\br,$$
where $|\pm s|=\s$. Take one of them, say $s$. This trivialization determines a contact 1-form $\zh$ such that
$$\zw=\xd s\we\zh+s\xd\zh,$$
and the Liouville 1-form $\zvy=i_\n\zw$ reads $\zvy=s\zh$. The submanifold $\tM$ is defined by $|s|=1$, and the Liouville 1-form restricted to $\tM$ is $\zh$ on the $s=1$ part, and $-\zh$ on the $s=-1$ part, defining, therefore, a contact form $\tzh$ on $\tM$.

Of course, $\tM$ does not depend on the choice of the sign before $s$, and the restriction of $\zvy$ to $\tM$ is uniquely defined, so also does not depend on this sign: if $s$ goes to $-s$, $\zh$ goes to $-\zh$, and on $-s=1$ we get $-\zh_s$, so $\tzh$ remains the same. By homogeneity, the diffeomorphisms $h_{-1}$ on $P$, given by the action of $-1\in\Rt$, changes $\zw$ to $-\zw$, so its restrictions $\zs$ to $\tM$ changes $\tzh$ into $-\tzh$.

\mn Finally, as the contact distribution $C$ is the projection of $\ker(\zvy)$ onto $M$, it is easy to see that the projection of $\ker(\zvy)$ gives $\tC=\ker(\tzh)$.

\end{proof}
\begin{remark}\label{ttriv} Since homogeneous structures on $P$ are, in a local trivialization, completely determined by any of two connected components of $\zt^{-1}(U)=U\ti\Rt$, we can use the calibration $\s$ to define an \emph{$\s$-trivialization} as the trivial $\R_+$-bundle $U\ti\R_+$ with the trivialization of each component $U\ni y\mapsto\bl\zt(y),\s(y)\br$. In other words, if a calibration is given, we can actually reduce our considerations to trivial $\R_+$-bundles.
\end{remark}
\no In the presence of a Riemannian metric, the idea of a contact metric structure also makes sense for a general contact manifold $(M,C)$.
\begin{definition}\label{defmai}
Let $g_M$ be a Riemannian metric on a contact manifold $(M,C)$ and let $(\tM,\tzh)$ be the induced covering of $(M,C)$. We call $g_M$ a \emph{Levi metric} if the pullback $g_\tM$ of $g_M$ to $\tM$ is a Levi metric for $(\tM,\tzh)$. In this case, we call $(M,C,g_M)$ a \emph{contact metric manifold}, and the cooriented contact metric manifold $(\tM,\tzh, g_\tM)$--the \emph{associated trivialization}.
\end{definition}\label{defmai1}
\no The corresponding split vector field $\tA$ and the split Levi structure $\tbp$ represent a paired vector field $|A|$ and a paired $(1,1)$-tensor field $|\bp|$ on $M$. Clearly, $\fn_\bp=\fn_{-\bp}$, since simultaneously $\zh$ goes to $-\zh$ and $A$ goes to $-A$, so the condition $\fn_{|\bp|}=0$ makes a perfect sense. On the level of $\tM$, it means that $\fn_\tbp=0$. This allows for a natural definition of a generalized Sasaki structure for distributional contact structures.
\begin{definition} Let $(M,C,g_M)$ be a contact metric manifold, and $(\tM,\tzh,g_\tM)$ the associated trivialization. Let $\zm$ be a 1-form on $M$, and $\tzm$ be the pullback of $\zm$ to $\tM$. We call $(M,C,g_M,\zm)$ a \emph{generalized Sasaki structure} if its trivialization $(\tM,\tzh,g_\tM,\tzm)$ is a (cooriented) {generalized} Sasaki structure. In other words, $\fn_{|\bp|}=0$ and
{$$\xd\zm(\pm A,X)=\zm\bl\mp\pc(X)\br,$$}
for $X\in C$. {In the case $\zm=0$, we just call $(M,C,g_M)$ a \emph{Sasaki structure}.}
\end{definition}
\no This way, we get the following.
\begin{theorem}\label{tmain} There is a canonical one-to-one correspondence between generalized Sasaki structures $(M,C,g_M,\zm)$ on a contact manifold $(M,C)$ and homogeneous K\"ahler structures on its symplectization $(P,\zw)$. In a local trivialization of $P\big|_U=U\ti\Rt$ associated with the $g_M$-calibration, the corresponding Riemannian metric reads
$$g(x,s)={|s|}\Bl\big(\xd s/s+\zm(x)\big)^2+ g_M(x)\Br.$$
\end{theorem}
\no The corresponding homogeneous K\"ahler structure is called the \emph{K\"ahlerianization} of the generalized Sasaki structures.

\section{Sasakian products}
The Cartesian product of two Sasakian manifolds is even-dimensional, so it cannot carry any contact structure. The \emph{contact product} of contact manifolds has one additional dimension. A canonical approach is related to an obvious product of symplectic $\Rt$-bundles, which are symplectizations of contact manifolds \cite{Arnold:1989,Bruce:2017,Grabowski:2013}. The appropriate definition of the contact product (see \cite{Grabowska:2026}) follows from a more general concept of a product of general Jacobi structures \cite{Vitagliano:2019}, understood as \emph{Poisson $\Rt$-bundles} \cite{Bruce:2017,DiCosmo:2025,Marle:1991,Vitagliano:2018}. This point of view was used in \cite{DiCosmo:2025} for an approach to the Jacobi sigma models.

\mn Instead of proposing an \emph{ad hoc} definition of a Sasakian product of generalized Sasaki manifolds, we will derive a conceptual approach to this question \emph{via} homogeneous K\"ahler structures on $\Rt$-bundles, since there is a canonical one-to-one correspondence between generalized Sasaki structures on a contact manifold $(M,C)$ and homogeneous K\"ahler structures on its symplectic cover $(P,\zw)$.

\mn It is well known that the Cartesian product of two (almost) K\"ahlerian manifolds is canonically again (almost) K\"ahler. The point is that if the ingredients are homogeneous on $\Rt$-bundles $P_i$, $i=1,2$, then the product is also homogeneous on the $\Rt$-bundle $P=P_1\ti P_2$ with the diagonal $\Rt$-action, thus defines a generalized Sasaki structure on the base--the product Sasaki structure we are looking for. A variant of this idea for quasi-regular Sasakian manifolds has been proposed in \cite{Boyer:2007} (see also \cite{Wang:1990}). Having distinguished calibrations, we have distinguished paired contact forms on the base manifolds; it is enough to reduce ourselves to the case of cooriented contact manifolds. In the case of cooriented contact structures, it is sufficient to consider $\R_+$-bundles instead of $\Rt$-bundles, and the corresponding product contact forms (more generally, the product of Jacobi structures) have already been defined in \cite{Ibanez:1997}. The contact product distribution is canonical, but the choice of the product contact form in the conformal class made in \cite{Ibanez:1997} seems to be done \emph{ad hoc}.

\mn To be more precise, let us start with the products of $\Rt$-principal bundles, $\zt_i:P_i\to M_i$, $i=1,2$, with the principal $\Rt$-actions $h^i$. The Cartesian product $P=P_1\ti P_2$ is again an $\Rt$-principal bundle, with the diagonal principal $\Rt$-action $h$,
\be\label{r-act}
h_t(y_1,y_2)=\big(h^1_{t}(y_1), h^2_{t}(y_2)\big).
\ee
The dimension of the base $M=M_1\ti^!M_2={(P_1\ti P_2)}/\Rt$ of this product's principal bundle is $\dim(M_1)+\dim(M_2)+1$. This product's principal bundle $\zt:P\to M$ was denoted
\be\label{prP}\zt:P_1\ti^! P_2\to M_1\ti^!M_2\ee
in \cite{Grabowska:2026}. The Liouville vector field $\n$ of this $\Rt$-action is the sum $\n=\n_1+\n_2$ of the Liouville vector fields $\n_i$ on $P_i$, $i=1,2$. Here and further, we use the convention in which a tensor field on factors of a Cartesian product is also interpreted as the corresponding tensor field on the product, trivially extended to the other factor.

\begin{remark}\label{rpr}The smooth manifold $M_1\ti^!M_2$ looks like an $\Rt$-principal bundle over $M_1\ti M_2$ \cite{Zapata-Carratala:2020}, {though} not in a strict sense. Actually, $M_1\ti^!M_2$ is a principal  $G=(\Rt\ti\Rt)/\Rt$-bundle, where $\Rt$ is embedded in the product as the subgroup of diagonal elements. This is because the quotient group $G$, although isomorphic to $\Rt$, does not have a privileged $\Rt$-parametrization. One can consider the parametrizations coming from the left or the right factor in $\Rt\ti\Rt$, but we will prefer quotients of parametrizations. They are not strictly canonical, as one cannot distinguish $t=s_1/s_2$ from $r=s_2/s_1$.
\end{remark}
\no Suppose now that the $\Rt$-principal bundles $P_1$ and $P_2$ are symplectic covers of contact manifolds $(M_i,C_i)$, $i=1,2$, with homogeneous symplectic forms $\zw_1$ and $\zw_2$. The Cartesian product $P_1\ti P_2$ carries the canonical symplectic form
$$(\zw_1+\zw_2)(y_1,y_2)=\zw_1(y_1)+\zw_2(y_2),$$
which is also homogeneous with respect to the diagonal $\Rt$-action. Hence, the product $M_1\ti^! M_2$ carries a canonical contact structure $C=C_1\tip C_2$, and we obtain a canonical \emph{contact product} $(M_1\tip M_2,C_1\tip C_2)$ of contact manifolds $(M_i,C_i)$, $i=1,2$.
\begin{remark}
In Section \ref{Sec2}, we made clear that there is a canonical correspondence between line bundles and $\Rt$-principal bundles, which defines an equivalence of the corresponding categories (see Remark \ref{prodlb}). Let $L_i\raa M_i$ be a line bundle corresponding to the $\Rt$-bundle $P_i$, $i=1,2$. As the product $P_1\ti^! P_2$ is again an $\Rt$-principal bundle, there exists the corresponding line bundle $L_1\ti^! L_2$, understood as the product in the category of line bundles. In \cite{Schnitzer:2023,Zapata-Carratala:2020}, the authors discuss this product and its properties in detail.
\end{remark}
\begin{example}\label{pks}
Take two cooriented contact manifolds $(M_i,\zh_i)$ and their  symplectic covers $(P_i,\zw_i)$, $i=1,2$, where $P_i=M_i\ti\R_+$ with the fiber coordinate $s_i$ and
\be\label{zw1}\zw_i(x_i,s_i)=\xd s_i\we\zh_i(x_i)+s_i\cdot\xd\zh_i(x_i).\ee
On the product $P_1\ti P_2$, introduce new coordinates
$$s=s_1+s_2, \quad t=\frac{s_1}{s_2}. $$
Then,
\be\label{sstit} s_1=\frac{ts}{1+t}, \quad s_2=\frac{s}{1+t}. \ee
The coordinate $s$ is 1-homogeneous for the diagonal $\R_+$-action, while $t$ is invariant. The coordinate $s$ is canonical as the fiber coordinate on the $\Rt$-bundle $P_1\tip P_2\to M_1\tip M_2$, while $t$ is not quite; the coordinate $r=s_2/s_1=1/t$ is equivalently good. This goes back to the problem of parametrization of the group $(\R_+\ti\R_+)/\R_+$ (cf. Remark \ref{rpr}). With $s$ there is associated a canonical contact form $\zh$ on $M=M_1\tip M_2$ \emph{via} the direct analog of (\ref{zw1}). We denote this form $\zh=\zh_1\tip\zh_2$. In coordinates $(x_1,x_2,t)$ on $M=M_1\tip M_2$, the product contact form reads
\be\label{cfor}\zh = \frac{t\zh_1+\zh_2}{1+t},\ee
and the product symplectic structure is $\zw=\xd(s\zh)$. Note that the above formula completely agrees (up to the parametrization) with the one in \cite[Proposition 3.5]{Ibanez:1997}.

\mn It is easy to check that the Reeb vector field of $\zh$ is $\zx=\zx_1+\zx_2$. The kernel $C$ of $\zh$ contains $C_1\op C_2$ and two additional generators, e.g.,
$$R=\frac{\zx_1-t\zx_2}{1+t}\quad\text{and}\quad Q=t\pa_t.$$
The vector field $Q$ is the generator of a residual $\R_+$-action
$$h_\zt(x_1,x_2,t)=(x_1,x_2,\zt t)$$
on $M$, which turns $M$ into an $\R_+$-bundle. This is not quite canonical (cf. Remark \ref{rpr}), as
$$h_\zt(x_1,x_2,r)=(x_1,x_2,\zt^{-1}r),$$
so there is no way to distinguish canonically one of $Q,-Q$. Anyhow, we end up with the product contact manifold
$$(M,\zh)=\bl M_1\tip M_2,\zh_1\tip\zh_2\br.$$

\mn Similarly, we consider products of homogeneous Riemannian metrics. Consider a homogeneous metric
$$g_i{=}s_i\Bl\bl\xd s_i/s_i+\zm_i\br^2+g_{M_i}\Br,$$
on $P_i=M_i\ti\R_+$, i=1,2.
The product metric $g=g_1+ g_2$ on $P_1\ti P_2$ reads
\be\label{gg}g=s_1\Bl\bl\xd s_1/s_1+\zm_1\br^2+g_{M_1}\Br+s_2\Bl\bl\xd s_2/s_2+\zm_2\br^2+g_{M_2}\Br.\ee
Substituting (\ref{sstit}) into $g$ (cf. (\ref{gg})), gives
\be\label{gfor}g = s\left[ \left( \frac{\xd s}{s} + \frac{t\zm_1+\zm_2}{1+t} \right)^2 + g_M \right].\ee
The induced metric $g_M$ on $M$ reads
\be\label{gmfor} g_M = \frac{t}{1+t}g_{M_1} + \frac{1}{1+t}g_{M_2} + \frac{t}{(1+t)^2} \left(
\frac{\xd t}{t} + \zm_1-\zm_2 \right)^2,\ee
and the new connection form reads
\be\label{mi}\zm = \frac{t\zm_1+\zm_2}{1+t}.\ee
With the coordinate $r=1/t$, we get their `symmetric' versions,
\bes
&\zh = \frac{\zh_1+r\zh_2}{1+r};\\
&g = s\left[ \left( \frac{\xd s}{s} +
\frac{\zm_1+r\zm_2}{1+r} \right)^2 + g_M \right];\\
&g_M = \frac{1}{1+r}g_{M_1} +
\frac{r}{1+r}g_{M_2} + \frac{r}{(1+r)^2} \left( \frac{\xd r}{r} + \zm_2-\zm_1 \right)^2;\\
&\zm = \frac{\zm_1+r\zm_2}{1+r}.
\ees
As we can see, the connection 1-forms $\zm_i$ generally interact with $g_M$, which forces the need for some deformations in comparison to the flat case: $\zm_1=0$, $\zm_2=0$, $A_1=\zx_1$, $A_2=\zx_2$.
\end{example}
\no Suppose now that $(\zw_i,g_i)$ are homogeneous almost K\"ahler structures on $P_i$ with almost complex structures $J_i$, $i=1,2$. Then $(\zw_1+\zw_2,g_1+ g_2)$ is a homogeneous almost K\"ahler structure on $P=P_1\ti P_2$ with the almost complex structure $J_1+J_2$. The base manifolds $M_i$ are split Levi manifolds $(M_i,\zh_i,\bar\phi_{C_i})$ with the Levi metrics $g_{M_i}$, the split vector fields $A_i$, and the connection 1-forms $\zm_i$, $i=1,2$.

\mn According to Theorem \ref{akt}, the product contact manifold $(M,\zh)$, where $\zh$ is given by (\ref{cfor}), carries a canonical structure of a split Levi manifold with the Levi metric (\ref{gmfor}) and the product connection (\ref{mi}).

By direct calculations, we get the product split vector field in the form
$$A=A_1+A_2+(f_2-f_1)Q,$$
where $f_i=\zm_i(A_i)$, $i=1,2$, so we have $\zh(A)=1$ and $i_Ag_M=\zh$. The product contact distribution is
$$ C=C_1\oplus C_2\oplus\langle R,Q\rangle,$$
where $Q=t\pa_t$ and
$$R=\frac{A_1-tA_2}{1+t}{-\frac{tf_2+f_1}{1+t}Q}.$$
The product split Levi structure is
\bel\label{pbp} & \bp(A)=0, \quad \bp(R)={-Q}, \quad \bp(Q)={R},\\
& \bp(X_1) = \pc_1(X_1)+\zm_1(X_1)R-\zm_1\bl\pc_1(X_1)\br Q,\nn\\
&\bp(X_2) =\pc_2(X_2)-\zm_2(X_2)R+\zm_2\bl\pc_2(X_2)\br Q,\nn
\eel
where $X_i\in C_i$, $i=1,2$.

\mn The corresponding Levi form $g_C(V,V')=\xd\zh\bl V,\pc(V')\br$ for
$$V=X+Y+aR+bQ,\quad V'=X'+Y'+a'R+b'Q,$$
where $X,X'\in C_1$ and $Y,Y'\in C_2$, reads
\bel\label{GG}&g_C(V,V') = \frac{t}{1+t}\,g_{C_1}(X,X') + \frac{1}{1+t}\,g_{C_2}(Y,Y') \\
&+\frac{t}{(1+t)^2}\Bl aa'+\bl b+\zm_1(X)-\zm_2(Y)\br\cdot\bl b'+\zm_1(X')-\zm_2(Y')\br\Br, \nn\eel
where $g_{C_i}$ is the Levi form of $\pc_i$, $i=1,2$. This is exactly the restriction of the metric (\ref{gmfor}) to $C$.

\begin{theorem}
If $(P_i,\zw_i,g_i)$ is the K\"ahlerianization of a {generalized Sasakian manifold} $(M_i,C_i,g_{M_i},\zm_i)$, $i=1,2$, then $(P_1\ti P_2,\zw_1+\zw_2,g_1+g_2)$ is the K\"ahlerianization of a uniquely determined {generalized} Sasakian manifold $(M=M_1\tip M_2,\zh,g_M,\zm)$, where $M_1\tip M_2=(P_1\ti P_2)/\Rt$. In local coordinates,
$$(x_1,x_2,t)\in M_1\ti M_2\ti\R_+,$$
on {one sheet the associated cover, with} $t=\s_1/\s_2$,
\bes & \zh(x_1,x_2,t) = \frac{t\zh_1(x_1)+\zh_2(x_2)}{1+t},\\
& g_M(x_1,x_2,t) = \frac{t}{1+t}g_{M_1}(x_1) + \frac{1}{1+t}g_{M_2}(x_2) + \frac{t}{(1+t)^2} \left(
\frac{\xd t}{t} + \zm_1(x_1)-\zm_2(x_2) \right)^2,\\
& \zm(x_1,x_2,t) = \frac{t\zm_1(x_1)+\zm_2(x_2)}{1+t}.
\ees
\end{theorem}
\no The above described Sasakian manifold $(M=M_1\tip M_2,{\zh},g_M,\zm)$ we call the \emph{Sasakian product} of the Sasakian manifolds $(M_i,C_i,g_{M_i},\zm_i)$, $i=1,2$. For {classical Sasaki manifolds}, we get the following.
\begin{example}
In the classical Sasaki case, $(M_i,\zh_i,g_{M_i})$, we have $A_i=\zx_i$ and $\zm_i=0$, $i=1,2$. Moreover,
$$g_{M_i}=\zh_i^2+g_{C_i},\qquad g_{C_i}=\xd\zh_i\circ(\id_{C_i}\ot\phi_{C_i}).$$
In this case, the product structures are:
$$\zh=\frac{t\zh_1+\zh_2}{1+t},\qquad A=\zx_1+\zx_2,\qquad R=\frac{\zx_1-t\zx_2}{1+t},\qquad Q=t\partial_t.$$
The product metric becomes
$$g_M = \frac{t}{1+t}g_{M_1} + \frac{1}{1+t}g_{M_2} + \frac{t}{(1+t)^2} \left(\frac{\xd t}{t}\right)^2.$$
Using
$$\frac{t}{1+t}\zh_1^2 + \frac{1}{1+t}\zh_2^2 = \zh^2 + \frac{t}{(1+t)^2}(\zh_1-\zh_2)^2,$$
we obtain
$$g_M= \zh^2 + \frac{t}{1+t}g_{C_1} + \frac{1}{1+t}g_{C_2} + \frac{t}{(1+t)^2} \left((\zh_1-\zh_2)^2 + \left(\frac{\xd
t}{t}\right)^2\right).$$
The product Levi structure can be expressed in the tensorial form as
$$\pc = \phi_{C_1}+\phi_{C_2} + R\otimes\frac{\xd t}{t} - Q\otimes(\zh_1-\zh_2).$$
\end{example}
\section{Conclusions and outlook}
Various approaches to Sasakian geometry have been proposed in the literature, but they have almost exclusively concerned cooriented contact manifolds, i.e., contact manifolds endowed with a globally defined contact form. In this paper, we introduced corresponding notions for arbitrary, possibly nontrivializable, contact manifolds $(M,C)$.

\mn Our approach is based on the general principle that contact structures, and more generally Jacobi structures, should be viewed as homogeneous symplectic or Poisson structures on principal $\Rt$-bundles \cite{Arnold:1989,Bruce:2017,Grabowska:2022,Grabowska:2023,Grabowski:2013}. Motivated by the classical description of Sasakian manifolds via K\"ahler geometry on the cone $\cM=M\times\R_+$, we developed a theory of homogeneous (almost) K\"ahlerian structures on principal $\Rt$-bundles.

\mn A key step in reducing the general situation to the classical cooriented setting was the observation that the metric $g_M$ canonically determines a class of local contact forms differing only by sign, and that this class is preserved under coordinate changes. Geometrically, this is a consequence of the fact that the line bundle orthogonal to the contact distribution $C\subset \sT M$ is canonically identified with the quotient line bundle $L=\sT M/C$. Equivalently, one obtains a canonical paired contact form $|\zh|$ on $M$. This leads naturally to a distinguished two-sheeted covering $\tM$ of $M$, equipped with a genuine global contact form.

\mn Using this reduction to the cooriented case, we obtained a complete characterization of (almost) K\"ahlerian $\Rt$-bundles in terms of principal connections on $P$ and geometric data on the underlying contact manifold. This provides a natural and conceptually well-founded definition of generalized Sasakian structures on arbitrary contact manifolds.

\mn Furthermore, employing a natural product construction in the category of principal $\Rt$-bundles, we introduced a corresponding notion of a Sasakian product of generalized Sasakian manifolds. A deeper study of the resulting category of generalized Sasakian manifolds, together with applications to canonical contact structures on first jet bundles of line bundles, will be pursued elsewhere. To the best of our knowledge, these questions have not previously been investigated.

\section{Acknowledgements}
We are grateful to the referee for a careful reading of the manuscript and for insightful comments and detailed suggestions that have substantially improved its quality. We also thank Luca Vitagliano for valuable remarks that contributed significantly to the final version of this work.






\vskip1cm
\noindent Katarzyna Grabowska\\\emph{Faculty of Physics,
University of Warsaw,}\\
{\small ul. Pasteura 5, 02-093 Warszawa, Poland} \\{\tt konieczn@fuw.edu.pl}\\
https://orcid.org/0000-0003-2805-1849\\

\noindent Janusz Grabowski\\\emph{Institute of Mathematics, Polish Academy of Sciences}\\
{\small ul. \'Sniadeckich 8, 00-656 Warszawa, Poland}\\
{\tt jagrab@impan.pl}\\  https://orcid.org/0000-0001-8715-2370\\

\no Rouzbeh Mohseni\\\emph{Faculty of Mathematics and Computer Science, University of Łódź}\\
{\small ul. Banacha 22, 90-238 Łódź, Poland}\\
{\tt rouzbeh.mohseni@wmii.uni.lodz.pl}\\ https://orcid.org/0000-0003-3773-6268


\begin{thebibliography}{V}
\bibitem{Arnold:1989} V.~I.~Arnold,
\newblock{Mathematical Methods of Classical Mechanics,}
\newblock{\emph{Graduate Texts in Mathematics} \textbf{60}, Springer-Verlag, New York, 1989.}

\bibitem{Blair:1976} D.~E.~Blair,
\newblock{On the non-existence of flat contact metric structures,}
\newblock{\emph{Tohoku Math. J.} \textbf{28} (1976), 373--379.}

\bibitem{Blair:2010} D.~E.~Blair,
\newblock{Riemannian geometry of contact and symplectic manifolds,}
\newblock{\emph{Progr. Math.} \textbf{203},
Birkhäuser Boston, Ltd., Boston, MA, 2010.}

{\bibitem{Boothby:1958} W. M. Boothby, H. C. Wang,
\newblock{On contact manifolds,}
\newblock{\emph{Ann. of Math.,} \textbf{68} (1958), 721--734.}
}

\bibitem{Boyer:2007} Ch.~P.~Boyer, K.~Galicki, L.~Ornea,
\newblock{Constructions in Sasakian geometry,}
\newblock{\emph{Math. Z.} \textbf{257} (2007), 907--924.}

\bibitem{Boyer:2008} Ch.~P.~Boyer, K.~Galicki,
\newblock{Sasakian geometry,}
\newblock{\emph{Oxford Math. Monogr.}, Oxford University Press, Oxford, 2008.}

\bibitem{Bruce:2017} A.~J.~Bruce, K.~Grabowska, J.~Grabowski,
\newblock{Remarks on contact and Jacobi geometry,}
\newblock{\emph{SIGMA Symmetry Integrability Geom. Methods Appl.} \textbf{13} (2017), Paper No. 059, 22 pp.}


\bibitem{Cappelletti:2015} {B.~Cappelletti Montana, A. De Nicola,} and I.~Yudin,
\newblock{Hard Lefschetz theorem for Sasakian manifolds,}
\newblock{\emph{J. Differential Geom.} \textbf{101} (2015), 47--66.}

\bibitem{DiCosmo:2025} F.~Di~Cosmo, K.~Grabowska, J.~Grabowski,
\newblock{Jacobi algebroids and Jacobi sigma models,}
\newblock{\emph{Rev. Math. Phys.} \textbf{37} (2025), 2550004 (42 pages).}

\bibitem{Geiges:2008} H.~Geiges,
\newblock{An introduction to contact topology,}
\newblock{\emph{Cambridge Stud. Adv. Math.} \textbf{109},
Cambridge University Press, Cambridge, 2008.}

\bibitem{Grabowska:2022} K.~Grabowska, J.~Grabowski,
\newblock{A novel approach to contact Hamiltonians and contact Hamilton-Jacobi {theory}.}
\newblock{\emph{J. Phys. A} \textbf{55} (2022), 435204 (34pp).}

\bibitem{Grabowska:2023} K.~Grabowska, J.~Grabowski,
\newblock{Reductions: precontact versus presymplectic,}
\newblock{\emph{Ann. Mat. Pura Appl. {(4)}} \textbf{202} (2023), 2803--2839.}

\bibitem{Grabowska:2023a} K.~Grabowska, J.~Grabowski,
\newblock{Contact geometric mechanics: the Tulczyjew triple,}
\newblock{\emph{Adv. Theor. Math. Phys.} \textbf{28} (2024), 599--654.}

{\cm\bibitem{Grabowska:2026} K.~Grabowska, J.~Grabowski,
\newblock{The regularity and products in contact geometry,}
\newblock{\emph{Ann. Mat. Pura Appl. {(4)}} \textbf{205} (2026), 993--1015.}
}
\bibitem{Grabowska:2024}	K.~Grabowska, J.~Grabowski, M.~Ku\'s, and {G.~Marmo,}
\newblock{Contactifications: a Lagrangian description of compact Hamiltonian systems,}
\newblock{\emph{J. Phys. A} \textbf{57} (2024), 395204 (31pp).}

\bibitem{Grabowski:2012a} J.~Grabowski,
\newblock{Modular classes of skew algebroid relations,}
\newblock{\emph{Transform. Groups} \textbf{17} (2012), 989--1010.}

{\bibitem{Grabowski:2013}J. Grabowski,
\newblock{Graded contact manifolds and contact Courant algebroids,}
\newblock{\emph{J. Geom. Phys.} \textbf{68} {(2013)}, 27--58.}}

\bibitem{Grabowski:2005} J.~Grabowski, M.~Ku\'s, G.~Marmo,
\newblock{Geometry of quantum systems: density states and entanglement,}
\newblock{\emph{J. Phys. A} \textbf{38} (2005), 10217--10244.}


\bibitem{Grabowski:2006} J.~Grabowski, G.~Marmo, {P.~W.~Michor,}
\newblock{Homology and modular classes of Lie algebroids,}
\newblock{\emph{Ann. Inst. Fourier {(Grenoble)}} \textbf{56} (2006), 69--83.}

\bibitem{Grabowski:2009} J.~Grabowski, M.~Rotkiewicz,
\newblock{Higher vector bundles and multi-graded symplectic manifolds,}
\newblock{\emph{J. Geom. Phys.} \textbf{59} (2009), 1285--1305.}

\bibitem{Grabowski:2012} J.~Grabowski, M.~Rotkiewicz,
\newblock{Graded bundles and homogeneity structures,}
\newblock{\emph{J. Geom. Phys.} \textbf{62} (2012), 21--36.}



\bibitem{Gray:1959} J. W. Gray,
\newblock{Some global properties of contact structures,}
\newblock{\emph{Annals of Mathematics} \textbf{69} (1959),  421--450.}

\bibitem{Greenfield:1968} S. Greenfield,
\newblock{Cauchy-Riemann equations in several variables,}
\newblock{\emph{Ann. Sc. norm. super. Pisa - Cl. sci.} \textbf{22} (1968), 275--314.}


\bibitem{Ianus:1972} S. Ianus,
\newblock{Sulle variet\`a di Cauchy-Riemann,}
\newblock{\emph{Rend. Accad. Sci. Fis. Mat. Napoli}, XXXIX (1972), 191--195.}

\bibitem{Ibanez:1997} R.~Ib\'a\~nez, M.~de León, J.~C.~Marrero, D.~Mart\'{\i}n de Diego,
\newblock{Co-isotropic and Legendre-Lagrangian submanifolds and conformal Jacobi morphisms,}
\newblock{\emph{J. Phys. A} \textbf{30} (1997), 5427--5444.}

\bibitem{Kirillov:1976} A.~A.~Kirillov,
\newblock{Local Lie algebras,}
\newblock{\emph{Russian Math. Surveys} \textbf{31} (1976), no. 4, 55--75.}


\bibitem{Lie:1890} S. Lie,
\newblock{Theorie der Transformationsgruppen Abschn. 2,}
\newblock{Leipzig: Teubner, 1890.}


\bibitem{Marle:1991} C.~M.~Marle,
\newblock{On Jacobi manifolds and Jacobi bundles,}
\newblock{in Symplectic geometry, groupoids, and integrable systems (Berkeley, CA, 1989), 227--246,
\emph{Math. Sci. Res. Inst. Publ.} \textbf{20}, Springer, New York, 1991.}


\bibitem{Newlander:1957}  {A.~Newlander, L.~Nirenberg,}
\newblock{Complex Analytic Coordinates in Almost Complex Manifolds,}
\newblock{\emph{Ann. Math. (2)} \textbf{65} (1957), 391--404.}

\bibitem{Ornea:2007} L.~Ornea, M.~Verbitsky,
\newblock{Sasakian structures on CR-manifolds,}
\newblock{\emph{Geom. Dedicata} \textbf{125} (2007), 159--173.}

\bibitem{Ornea:2024} L.~Ornea, M.~Verbitsky,
\newblock{Principles of locally conformally Kähler geometry,}
\newblock{\emph{Birkhäuser/Springer, Cham}, 2024.}

\bibitem{Sasaki:1961} S. Sasaki, Y. Hatakeyama,
\newblock{On differentiable manifolds with certain structures which are closely related to almost contact structure. II,}
\newblock{\emph{Tohoku Mathematical Journal, Second Series} \textbf{13} (1961), no. 2, 281--294}.

\bibitem{Schnitzer:2023} J. Schnitzer, A. G. Tortorella,
\newblock{Weak dual pairs in Dirac–Jacobi geometry}
\newblock{\emph{Commun. Contemp. Math.} \textbf{25}, no. 08 (2023), 2250035.}

{\bibitem{Sekiya:2015} K. Sekiya,
\newblock{Generalized Almost Contact Structures and Generalized Sasakian Structures,}
\newblock{\emph{Osaka J. Math.} \textbf{52} (2015), 43--59.}}

\bibitem{Tanno:1989} S. Tanno,
\newblock{Variational Problems on Contact Riemannian Manifolds,}
\newblock{\emph{Trans. Amer. Math. Soc.} \textbf{314} (1989), {349--379}.} 

\bibitem{Tashiro:1963} Y. Tashiro,
\newblock{On contact structure of hypersurfaces in complex manifolds I,}
\newblock{\emph{Tohoku Math. J.} \textbf{15} (1963), 62-78.}

\bibitem{Tortorella:2020} A.~G.~Tortorella, L.~Vitagliano, O.~Yudilevich,
\newblock{Homogeneous G-structures,}
\newblock{\emph{Ann. Mat. Pura Appl. {(4)}} (4) \textbf{199} (2020), 2357--2380.}

\bibitem{Vaisman:1976} I. Vaisman,
\newblock{On locally conformal almost Kähler manifolds,}
{\emph{Israel J. Math.} \textbf{24} (1976), 338--351.}

\bibitem{Vaisman:1980} I. Vaisman,
\newblock{Conformal changes of almost contact metric structures,}
\newblock{In: Artzy, R., Vaisman, I. (eds) Geometry and Differential Geometry. Lecture Notes in Mathematics, vol 792. \emph{Springer}, Berlin, Heidelberg, 1980.}

\bibitem{Vitagliano:2018} L.~Vitagliano,
\newblock{Dirac-Jacobi bundles,}
\newblock{\emph{J. Symplectic Geom.} \textbf{16} (2018), 485--561.}

\bibitem{Vitagliano:2019} L.~Vitagliano,
\newblock{Products in Jacobi Geometry,}
\newblock{conference talk, Rio de Janeiro, 2019.}

\bibitem{Wang:1990} {M.~Y.~Wang,} W.~Ziller,
\newblock{Einstein metrics on principal torus bundles,}
\newblock{\emph{J. Differential Geom.} \textbf{31} (1990), 215--248.}

\bibitem{Zapata-Carratala:2020} C.~Zapata-{Carratal\'a},
\newblock{Jacobi Geometry and Hamiltonian Mechanics: the Unit-Free Approach,}
\newblock{\emph{Int. J. Geom. Methods Mod. Phys.} \textbf{17} (2020), 2030005.}


\end{thebibliography}
\end{document}